\documentclass[reqno]{amsart}
\usepackage[utf8]{inputenc}
\usepackage{enumitem}

\usepackage{amsmath,amssymb,amsbsy,amsfonts,amsthm,latexsym,amsopn,amstext,mathtools,leftidx,
            amsxtra,euscript,amscd,verbatim,epsf,colordvi,graphics}
\usepackage{ytableau}
\usepackage{color}
\usepackage{pgf,tikz}
\usetikzlibrary{decorations.pathreplacing}
\usetikzlibrary{arrows}
\usepackage{mathrsfs}
\usepackage{graphicx}

\usepackage{hyperref}

\usepackage{cleveref}
\usepackage[margin=11pt,font=footnotesize,labelfont=bf]{caption}
\newtheoremstyle{tplain}{3pt}{3pt}{\rmfamily}{}{\bfseries}{.}{0.5em}{}
\theoremstyle{tplain}

\usepackage{pgf,tikz}
\usepackage{tcolorbox}
\usepackage[enableskew]{youngtab}
\definecolor{darkgreen}{cmyk}{1,0,1,0}
\newtheorem{thm}{Theorem}
\newtheorem{lem}{Lemma}
\newtheorem{ex}{Example}
\newtheorem{cor}{Corollary}
\newtheorem{prop}{Proposition}
\newtheorem{obs}{Remark}
\newtheorem{defi}{Definition}
\def \cal{\mathcal}
\def \rm {\mathrm}
\def \mbf {\mathbf}

\def\rev{{\rm rev}}
\def\Z { \mathbb{Z}}
\def\C { \mathbb{C}}
\def\LRAII {\mathsf{LR}}
\def \r {\mathbf{r}}
\def \t {\mathbf{t}}
\def \redu {\mathrm{red}}
\def\c { \mathrm{c}}
\def \k {\mathfrak{k}}
\def\wt { \mathrm{wt}}
\def\shape { \mathrm{shape}}
\def\rem { \mathrm{rem}}
\def\suc { \mathrm{suc}}
\def \s {\mathbf{s}}

\newcommand{\oo}{\color{blue}}
\newcommand{\blue}{\color{blue}}
\newcommand{\red}{\color{red}}
\newcommand{\green}{\color{green}}
\newcommand{\bro}{\color{brown}}

\newcommand{\ora}{\color{orange}}

\makeatletter
\newcommand*\bigcdot{\mathpalette\bigcdot@{.5}}
\newcommand*\bigcdot@[2]{\mathbin{\vcenter{\hbox{\scalebox{#2}{$\m@th#1\bullet$}}}}}
\makeatother

\newcommand*\circled[1]{\tikz[baseline=(char.base)]{
            \node[shape=circle,draw,inner sep=1pt] (char) {#1};}}

\newcommand{\YT}[3]{
\vcenter{\hbox{
\begin{tikzpicture}[x={(0in,-#1)},y={(#1,0in)}] 
\foreach \rowi [count=\i] in {#3} {
 \foreach \e [count=\j] in \rowi {
  \draw (\i,\j) rectangle +(-1,-1);
  \draw (\i-0.5,\j-0.5) node {$#2\e$};
 }
}
\end{tikzpicture}
}}
}

\title[Inverse reduction map] {
The inverse  reduction map  in the quantum Littlewood-Richardson bijection }

\author{Olga Azenhas}
\address{ University of Coimbra, CMUC, Department of Mathematics, Portugal}
\email{oazenhas@mat.uc.pt}



\begin{document}

\begin{abstract} We decompose  a symplectic column by detaching  its  maximal nonempty intervals that are fixed by the parity involution, this  is the part of the symplectic column that is fixed by that involution, and the remaining  part consisting of the subword whose image under the parity involution is disjoint of the symplectic column. The parity involution swaps an even number with the previous odd number and an odd number with the next even number. That decomposition is uniquely determined by the symplectic column and the symplectic conditions satisfied by those pieces pack the needed information to explicitly write the inverse of the reduction map  in the quantum Littlewood-Richardson bijection.
The tools provided here can be combined with the composition of the inverses of the  several maps in which the reduction map decomposes, given by Watanabe, namely, among them, combinatorial $R$-matrices and reduction maps  on shorter symplectic columns.
\end{abstract}
\maketitle

\tableofcontents

\section{Introduction}
Given   $\lambda \in Par_{\le 2n}$  a partition with at most $2n$ parts,
the quantum Littlewood-Richardson map $LR^{AII}$ \cite[Theorem 3.1.4]{watanabe}
is an one-to-one
assignment of a semi-standard tableau $T$ of shape $\lambda$  to a pair $(P^{II}(T), Q^{II}(T)) $   consisting of a symplectic tableau in $ SpT_{2n}(\mu)$,  for some shape $ \mu\in Par_{\le n}$, a partition with at most $n$ parts, and   a recording tableau of skew shape $\lambda/\mu$ in $Rec_{2n}(\lambda/\mu)$,

\begin{align} \label{lrsII}
 SST_{2n}(\lambda) \overset{\sim}\longrightarrow
\bigsqcup_{\begin{smallmatrix}\mu\in Par_{\le n}\\
\mu\subseteq\lambda\\
Q\in Rec_{2n}(\lambda/\mu)\end{smallmatrix}}SpT_{2n}(\mu) \times \{Q\}.
\end{align}

The set of recording tableaux $Rec_{2n}(\lambda/\mu)$ in the quantum Littlewood-Richardson (LR) map $LR^{AII}$ is in natural bijection \cite{watanabe, azreco} with the set of LR-Sundaram tableaux  $LRS_{2n}(\lambda/\mu)$ \cite{sundaram},
\begin{align}\label{record}Rec_{2n}(\lambda/\mu) \overset{\sim}\longrightarrow LRS_{2n}(\lambda/\mu).
\end{align}

For $T\in  SST_{2n}(\lambda)$, the computation of $P^{II}(T)$ is described as follows.  Decompose $T$, $d(T):=({\bf a}, T')$, where ${\bf a}\in SST_{2n}(\varpi_l)$  is the first column of $T$, $l=\ell(\lambda)$ the length of $\lambda$, and $T'$ is the semi-standard tableau in $  SST_{2n}(\lambda- \varpi_l)$, that remains after detaching the first column. Apply the reduction map, $\redu$,  to $\bf a$, to get $\redu({\bf a})\in SpT_{2n}(\varpi_t$  for some
$t\in[0,n]$ such that $0\le t\le min\{l,2n-l\}$ and $l-t\in 2\Z$ \cite[Proposition 4.3.6.]{watanabe}. The reduction map is injective \cite[Corollary 4.4.3]{watanabe} and can be computed using the removal operation \cite[Definition 2.5.1]{watanabe} with their nice properties \cite[Proposition 4.2.2, 4.2.7,  4.2.8]{watanabe} in the sense that $\redu({\bf a})={\bf a} \setminus \rem({\bf a})$.  Then by column Schensted insertion \cite{fulton, stanley}, insert $\redu({\bf a})$ into $T'$, $(T'\leftarrow \redu({\bf a}))$, to get the successor of $T$, $\suc(T)$,
\begin{align}T\mapsto d({\bf a}, T')\mapsto (\redu( {\bf a}), T')\mapsto \suc(T):=T'\otimes \redu({\bf a})=(T'\leftarrow\redu({\bf a}) ).
\end{align}

If $\redu({\bf a})\otimes T'$ is not a symplectic tableau, iterate the previous procedure until getting a symplectic tableau which is precisely $P^{II}(T)$. (If  $\redu({\bf a})\otimes T'$ is not symplectic, it means the first column is not symplectic, and the next successor will have  strictly less   cells otherwise the successor stabilizes. Thus the procedure  terminates after a finite number of iterations. The set $\rem({\bf a})$ consists of certain  pairs of entries, an odd entry followed  with an even entry (see Proposition \ref{prop:removcontain}, to be removed from $\bf a$ so that the resulting column is symplectic.  Thus the inverse of the reduction map consists in adding such pairs to a symplectic column so that they can be removed.

Recalling that the parity involution $\s$, Definition \ref{s}, swaps an even number with the previous odd number and an odd number with the next even number. In Theorem \ref{thm:symplecticdecomposition} we show that
any symplectic column ${\bf a}\in SpT_{2n}(\varpi_t)$ can be decomposed, for some $k\ge 0$, in the form
\begin{align}\label{introdu:symplecticolumn}
{\bf a}=
X^{(0)}{\bf A_1} X^{(1)}{\bf A_2}\cdots {\bf A_{k-1}}X^{(k-1)} {\bf A_k} X^{(k)},
\end{align}

where
\begin{itemize}
\item  for $i=0,1,\dots,k$, each $X^{(i)}$ is a column word of length $p_i\ge 0$, such that

\begin{align}\label{introdu:xx}\s(X^{(0)}X^{(1)}\cdots X^{(k)})\cap {\bf a}=\emptyset, \mbox{ and}
\end{align}

\item for $i=1,\dots, k$, each factor ${\bf A_i}$ is
a maximal nonempty interval  of even length $t_i\ge 2$     such that
\begin{align}\s({\bf A_i})={\bf A_i},
\end{align}
with $$t=\sum_{i=1}^kt_i +\sum_{i=0}^k p_i  \in [n].$$
\end{itemize}

For $i=1,\dots,k$, such intervals are symplectic pieces of the form
$${\bf A_i}=(a_i, a_i+1,\dots,a_i+t_i-1), \mbox{ $a_i\notin 2\Z$}$$
satisfying
\begin{align}
&a_i\ge 2t^{(i-1)}+t_i+1, ~\mbox{ with $\displaystyle t^{(i-1)}:=\sum_{j=1}^{i-1}t_j+\sum_{j=0}^{i-1}p_j$, $t^{(0)}:=p_0$, and $t^{(k)}:=t$.}
   \end{align}
Lemma \ref{lem:swapping} provides  equivalent useful ways to phrase the property \eqref{introdu:xx} to constructing the inverse reduction of the symplectic column word
$X^{(0)}X^{(1)}\cdots X^{(k)}$.
The inverse reduction of the symplectic column $\bf a$ in \eqref{introdu:symplecticolumn}, $\redu^{-1}({\bf a})$, is given in the following theorems as follows: Theorem \ref{thm:noconsecutive} gives $\redu^{-1}({\bf a})$ for $k=0$, that is, for ${\bf a}=X^{(0)})$ and so ${\bf a}\cap \s({\bf a})=\emptyset$;
Theorem \ref{thm:intervals} gives $\redu^{-1}({\bf a})$ for $p_0=\cdots=p_k=0$, that is, for ${\bf a}=\bigoplus_{i=1}^k {\bf A_i}$; and
Theorem \ref{thm:reductiexp} mixes the two previous theorems and gives $\redu^{-1}({\bf a})$ for any symplectic column $\bf a$. For the sake of legibility also some intermediate cases are provided, and all theorems are illustrated.

A compromise is a combination of our procedures with the inverse of  the decomposition of the reduced map into several maps namely combinatorial $R$-matrices and reduction maps  of lower degree \cite[Corollary 4.4.3]{watanabe}  as discussed in the last Section \ref{sec:final}.

\subsection{Organization} The paper is organized into four sections. Section \ref{sec:containmentmain} is devoted to provide the decomposition of symplectic columns, under the action of the parity involution Definition \ref{def:s}, in Theorem \ref{thm:symplecticdecomposition}.
Section \ref{sec:inverseredu} is devoted to providing a full procedure in Theorem \ref{thm:reductiexp} to compute the inverse reduction map on a symplectic column. The full procedure  runs across theorems \ref{thm:noconsecutive}, \ref{thm:xA}, \ref{thm:xAy}, \ref{lem:reductionsurj2}, and \ref{thm:intervals}  with many illustrations. Last section discusses our methods and its combination with combinatorial $R$-matrices.

\section*{Acknowledgements}
The author acknowledges financial support by the Centre for Mathematics of the University of Coimbra (CMUC, https://doi.org/10.54499/UID/00324/2025) under the Portuguese Foundation for Science and Technology (FCT), Grants UID/00324/2025 and UID/PRR/00324/2025.

\section{Symplectic tableaux and symplectic columns}\label{sec:containmentmain}

 Through out we fix the following notation. Given $a,b\in \Z$, $[a,b]:=\{x\in \Z: a\le x\le b\}$. When $a\in\Z^+$, we just write $[a]:=[1,a]$.

\subsection{Symplectic tableaux}
We fix $n\in\mathbb{N}$.
We first need the definition of the \emph{parity} (\emph{swapping}) \emph{involution} which plays a fundamental role in the analysis of symplectic columns and consequently later in the reduction map.

\begin{defi}\label{def:s}For each $x\in\mathbb{Z}$, set \cite{watanabe} \begin{align}\label{s} \s(x)=\begin{cases}
x+1, &  \mbox{ if } x\notin 2\mathbb{Z}, \\
x-1, &  \mbox{ if }  x\in 2\mathbb{Z}.
\end{cases}
\end{align}
\end{defi}
Indeed $\s^2(x)=x$, for each $x\in \Z$. We call it the parity (swapping) involution which has the following properties
\begin{align}\label{minmax}
min(a_i,\s(a_i))\notin 2\Z,~~min(a_i,\s(a_i))-1\in 2\Z, \\
 max(a_i,\s(a_i))\in 2\Z, ~~ max(a_i,\s(a_i))+1\notin 2\Z.
\end{align}

\begin{defi}\cite{king76} \label{def:symp}A semistandard tableau $G \in SST_{2n}(\lambda)$ is said to
be symplectic if
$$G(k, 1) \ge 2k-1, \mbox{  for all $  k \in  [ \ell(\lambda)]$}.$$
Let $SpT_{2n}(\lambda)$ denote the set of all symplectic tableaux of shape $\lambda$ on the alphabet $[2n]$.
\end{defi}

\begin{prop} \label{prop:symplectic1}  
Let $G \in SST_{2n}(\lambda)$.
\begin{enumerate}
\item\cite{watanabe} If $SpT_{2n}(\lambda)\neq\emptyset$ then $\ell(\lambda)\le n$.
\item If $G=(a_1,\dots,a_t)\in SpT_{2n}(\varpi_t)$ then $t\le n$ and
$t\le \lfloor\frac{a_t+1}{2}\rfloor$.
\item \cite{watanabe} If $G$ is not symplectic, then there exists a unique
$i \in[2, 2n]$ such that
\begin{align}G(i, 1) < 2i -1 \mbox{ and } G(k, 1) \ge 2k -1  \mbox{ for all $k \in [1, i -1 ]$}.\label{G1}\end{align}
Moreover, we have
\begin{align}G(i - 1, 1) = 2i - 3=2(i-1)-1  \mbox{  and $G(i, 1) = 2i -2=2(i-1)$ }.\label{G2}\end{align}
\item If the first column $G$ has no consecutive integers consisting of an odd number followed with an even number, that is, the first column of $G$ does not contain any interval of the form $[x<\s(x)]$, then  $G$ is symplectic.
\end{enumerate}
\end{prop}
\begin{proof} Points $(1)-(3)$ have been considered in \cite{watanabe} and \cite{azreco}.

Point $(4)$ follows from point $(3)$. Conditions \eqref{G1} and \eqref{G2} say that the first symplectic fail needs to be checked  in the first column of $G$ among consecutive integer entries consisting of an odd number followed with an even number, and if the fail occurs it happens for the first time in a such even entry. Therefore, if intervals $[x<\s(x)]$ do not occur in the first column of $G$, $G$is symplectic.
\end{proof}

\begin{cor}\cite{azreco}\label{symplecticcolumn} The following holds:
\begin{enumerate}
\item[(a)] For $0\le t\le n$, $SpT_{2n}(\varpi_t)\neq \emptyset$. Namely,  $SpT_{2n}(())=\{()\}$, and,  for $1\le t\le n$, $G=(1,3, \dots,2t-1)\in Sp_{2n}(\varpi_t)$ or $H=(2,4,\dots, 2t)\in Sp_{2n}(\varpi_t)$. In particular, $SpT_{2n}(\varpi_1)=SST_{2n}(\varpi_1)$.
\item [(b)]  $SpT_{2n}(\lambda)\neq\emptyset$ if and only if  $\ell(\lambda)\le n$.

\item [(c)] If $G'$ is obtained from $G\in  Sp_{2n}(\varpi_t)$ by suppressing
$0\le t_0\le t$ entries then $G'\in Sp_{2n}(\varpi_{t-t_0})$ and the suppressed part $G''\in Sp_{2n}(\varpi_{t_0})$.

 \item [(d)]
  Let $\lambda\in Par_\le n$,
 $$SpT_{2n}(\lambda)=\{S\in SST_{2n}(\lambda)| (S(1,1),\cdots,S(\ell(\lambda),1)\in SpT_{2n}(\varpi_{\ell(\lambda)}) \}.$$
where $S(i,1)$ indicates the entry  in row $i$ and column $1$ of $S$, for $1\le i\le \ell(\lambda)$.
 \end{enumerate}
\end{cor}

\subsection{Decomposition of symplectic columns under the action of the parity involution}
 Next lemmas are a direct consequence of the parity involution $\s$ \eqref{s}. They list fundamental properties for the inverse of the reduction map $\redu$ of a symplectic column. Often  a column $\mathbf{a}\in SST_{2n}(\varpi_t)$ is regarded as a set.

 \begin{lem}\label{interval}Let $\mathbf{a}=(a, a+1,\dots, a+t-1) \in SST_{2n}(\varpi_t)$ be an interval of length $t>0$. Then
  $s({\bf a})={\bf a}\Leftrightarrow $  $2\le t\in 2\Z$ and $a\notin 2\Z$.
 \end{lem}
 \begin{proof} Suppose $\s({\bf a})={\bf a}$. If $a\in 2\Z$ then $\s(a)=a-1\notin \bf a$, and  if $a+t-1\notin 2\Z$ the $\s(a+t-1)=a+t\notin \bf a$. Both are contradictions. Hence $a\notin 2\Z$ and $a+t-1\in 2\Z\Leftrightarrow t\in 2\Z$.
 The other implication is obvious.
 \end{proof}

\begin{defi}\label{maximalintervaldecompose}Let $\mathbf{a} \in SST_{2n}(\varpi_t)$, $t>0$. We write
\begin{align}{\bf a}=\bigoplus_{i=1}^k {\bf A}_i, \mbox{ for some $k\ge 1,$ }
  \end{align}
  to mean ${\bf a}=\bigsqcup_{i=1}^k {\bf A}_i$,  for some $k\ge 1,$ where  ${\bf A_i}=(a_i,a_i+1,\dots,a_i+t_i-1)\in SST_{2n}(\varpi_{t_i})$ with  $a_i\notin 2\Z$ and  $2\le t_i\in 2\Z$, such that $a_{i+1}-a_i\ge t_i+2$  for $i=1,\dots, k$, and $t=t_1+\cdots+t_k\in  2\Z$.

  In this case we say that $\bf a$ has a decomposition into maximal intervals of even length $\ge 2$ starting with an  odd  entry and ending with an even entry. Equivalently, we say that $\bf a$ has a decomposition into maximal nonempty intervals ${\bf A}_i$ that are fixed by the parity involution $\s$, $\s({\bf A}_i)={\bf A}_i$ for every $i$.
\end{defi}
\medskip
\begin{obs}${\bf a}=(4,5,6)(9,10,11)$ is decomposed into maximal intervals but they are not fixed by the parity involution, because their lengths are odd, $\s(456)=(3456)$ and $\s(9,10,11)=(9,10,11,12)$.
\end{obs}
Next lemma generalizes the previous one and  shows that $s({\bf a})=\bf a$ if and only if  $\bf a$ decomposes into maximal intervals ${\bf A}_i$ of even length $\ge 2$ starting with an  odd  entry and ending with an even entry. That is, it provides a necessary and sufficient condition for
a column ${\bf a}\in SST_{2n}(\varpi_t)$, $t>0$ to have such  decomposition.
 \begin{lem}\label{intervals}
 Let $\mathbf{a} \in SST_{2n}(\varpi_t)$, $t>0$. Then
  $s({\bf a})={\bf a} $  if and only if
  \begin{align}\label{intervaldecompose}{\bf a}=\bigoplus_{i=1}^k {\bf A}_i, \mbox{ for some $k\ge 1,$ }
  \end{align}
  with $\bf A_i$ an interval in the conditions of Definition \ref{maximalintervaldecompose}, for $i=1,\dots,k$.
 \end{lem}
 \begin{proof} From Lemma \ref{interval} the "if part" is obvious.

 We prove the "only if part". Let ${\bf a}=(x_1,x_2,\dots, x_t)$ such that $s({\bf a})=\bf a$. From Lemma \ref{interval}, $x_1\notin 2\Z$ and $x_2=s(x_1)=x_1+1\in {\bf a}\cap 2\Z$, $x_3\notin 2\Z$ and $x_4=s(x_3)=x_3+1\in {\bf a}\cap 2\Z$ and, more generally,

 $$ x_q\notin 2\Z, \mbox{ for $q\notin 2\Z$ and } s(x_q)=x_q+1=x_{q+1}\in 2\Z.$$

 Thus, $t\in 2\Z$,
 \begin{align}{\bf a}=\{x_1<s(x_1)=x_2=x_1+1,\, x_3<s(x_3)=x_3+1,\dots,\, x_{t-1}<s(x_{t-1})=x_{t-1}+1\}=\bigsqcup_{i\notin [t]\cap 2\Z}[x_i,\s(x_i)].
 \end{align} 

 Lastly, if $a_i\in \bf A_i$ and $a_{i+1}\in {\bf A_{i+1}}$ are  such that
 $a_{i+1}>a_i+t-1+1=a_i+t\notin 2\Z$ then
 $$a_{i+1}-(a_i+t)\ge 2\Leftrightarrow
 a_{i+1}-a_i\ge t+2.$$
 Therefore if $s({\bf a})=\bf a$, $\bf a$ decomposes into maximal intervals of even length $\ge 2$ starting with an odd entry and ending with an even entry.
 \end{proof}

\begin{lem}\label{lem:swapping}Let $\mathbf{a} = (a_1, \dots , a_t)\in SST_{2n}(\varpi_t)$.
 The following assertions are pairwise equivalent
\begin{enumerate}
\item for every $1\le i\le t$, $\s(a_i)\notin \bf{a}$.
\item ${\bf{a}}\cap \s({\bf{a}})=\emptyset$.
 \item  $[x, y]$ is an interval factor $\subseteq \bf a $ $\Rightarrow$ $x=y$  $ \vee$ ($x\in 2\Z, y=x+1\notin 2\Z$). \label{intervaltwo}
 \item for every $1\le i<t$,
$a_i\notin 2\Z\Rightarrow a_{i+1}-a_i\ge 2$ ( $\Leftrightarrow a_{i+1}\ge \s(a_i)+1$). \label{interval3}
\item for every $1< i\le t$,
$a_i\in 2\Z\Rightarrow a_{i}-a_{i-1}\ge 2$ ( $\Leftrightarrow \s(a_i)\ge a_{i-1}+ 1$).\label{interval4}
\item for every $1\le i<t$, \label{interval5}
\begin{align*}
&(  a_i\notin 2\mathbb{Z}, ~~ a_{i+1}\in 2\mathbb{Z}\Rightarrow a_{i+1}-a_i\ge 3)\\
&\qquad\qquad\wedge\\
&( a_i,~ ~a_{i+1}\notin 2\mathbb{Z} \mbox{ or } a_i, ~~a_{i+1}\in 2\mathbb{Z}\Rightarrow a_{i+1}-a_i\ge 2)\\
 &\qquad\qquad\wedge\\
 &( a_i\in 2\mathbb{Z},~ ~a_{i+1}\notin 2\mathbb{Z}\Rightarrow a_{i+1}-a_i\ge 1).
 \end{align*}
\end{enumerate}
\end{lem}
\begin{proof} $(1)\Leftrightarrow (2)$ is obvious.

$\eqref{intervaltwo} \Rightarrow (1)$. Let $x\in \bf a$. If $x\notin 2\Z$ then $s(x)=x+1\in 2\Z$ but the interval $[x,\s(x)]$ is not allowed in $\bf a$. Hence $s(x)\notin \bf a$. If $x\in 2\Z$ then $\s(x)=x-1\notin 2\Z$ but the interval $[\s(x), x]$ is not allowed in $\bf a$. Hence $s(x)\notin \bf a$.

$ (1)\Rightarrow\eqref{intervaltwo}$. Let $I=[x,y]\subseteq \bf a$. If $x\notin 2\Z$ then $s(x)=x+1\notin \bf a$. Hence $x=y$.
If $y\in 2\Z$ then $s(y)=y-1\notin \bf a$. Hence $x=y$. Therefore it remains the interval $I=\{x, x+1\}$ with $x\in 2\Z$.

\medskip
$(4)\Leftrightarrow(1).$ Let $i\in[t]$, $t>1$.

If   $a_i\notin 2\Z$ then $\s(a_i)=a_i+1\notin {\bf a}\Leftrightarrow
 a_{i+1}\ge \s(a_i)+1\Leftrightarrow a_{i+1}-a_i\ge 2$.

 It remains to consider the case $a_i\in 2\Z$.

 If   $a_i\in 2\Z$ and $i=1$ obviously $\s(a_1)=a_1-1\notin \bf a$.

 If $a_i\in 2\Z$ and $2\le i\le t$,
 then $\s(a_i)=a_i-1\notin 2\Z$. Moreover,  if $a_{i-1}=\s(a_i)=a_i-1\notin 2\Z$ then by $(4)$, $a_{i} -a_{i-1}\ge 2\Leftrightarrow 1\ge 2$ a contradiction. Therefore
 $$\s(a_i)\notin {\bf a}\Leftrightarrow a_{i-1}\le a_i-2\Leftrightarrow a_i-a_{i-1}\ge 2$$

\medskip
 $(5)\Leftrightarrow(1).$ Let $1<i\in[t]$, $t>1$.

 If   $a_i\in 2\Z$ then $\s(a_i)=a_i-1\notin {\bf a}\Leftrightarrow
 a_{i-1}\le \s(a_i)-1\Leftrightarrow a_{i}-a_{i-1}\ge 2$.

 It remains to consider the case $a_i\notin 2\Z$.

 If $i=t$ and $a_i\notin 2\Z$ then $\s(a_t)=a_t+1\notin\bf a$.
 If
$a_i\notin 2\Z$ and $1\le i<t$,
 then $\s(a_i)=a_i+1\in 2\Z$. Moreover,  if $a_{i+1}=\s(a_i)=a_i+1\in 2\Z$ then by $(5)$, $a_{i+1} -a_{i}\ge 2\Leftrightarrow 1\ge 2$ a contradiction. Therefore $\s(a_i)\notin {\bf a}\Leftrightarrow a_{i+1}\ge a_i+2\Leftrightarrow a_{i+1}-a_i\ge 2$.

\medskip
$(1)\Leftrightarrow \eqref{interval5}$ If $a_i\notin 2\mathbb{Z}$, and $ a_{i+1}\in 2\mathbb{Z}$ then
$$2\Z\ni\s(a_i)=a_i+1\notin {\bf a}\Leftrightarrow 2\Z\ni a_{i+1}\ge \s(a_i)+2\Leftrightarrow a_{i+1}-a_i\ge 3$$

If $ a_i$, $a_{i+1}\notin 2\mathbb{Z}$, then
$$2\Z\ni\s(a_i)=a_i+1\notin {\bf a}\Leftrightarrow 2\Z\reflectbox{$\notin$} a_{i+1}\ge \s(a_i)+1\Leftrightarrow a_{i+1}-a_i\ge 2$$

If $ a_i$, $a_{i+1}\in 2\mathbb{Z}$, then
$$ 2\Z\reflectbox{$\notin$}\s(a_{i+1})=a_{i+1}-1\notin {\bf a}\Leftrightarrow 2\Z\ni a_{i}\le \s(a_{i+1})-1\Leftrightarrow a_{i+1}-a_i\ge 2$$

The implication  $a_i\in 2\mathbb{Z}, a_{i+1}\notin 2\mathbb{Z}\Rightarrow a_{i+1}-a_i\ge 1$ is tautology.
\end{proof}

 We know from Proposition \ref{prop:symplectic1}, $(4)$,  that if in a column does not occur an odd number followed with an even number then the column is symplectic. This is what this lemma asserts lemma asserts in point $(3)$. So this characterizes  the  symplectic columns where the interval factors of length $>1$   if they exist, consist solely of an even number followed with an odd number. The lemma also provides several useful ways to phrase this property. Next corollary, concludes this.

\begin{cor}\label{cor:withoutA}If $\mathbf{a} \in SST_{2n}(\varpi_t)$  satisfies some assertion in Lemma \ref{lem:swapping} then $\mathbf{a} $ is symplectic and $\mathbf{a} \in SpT_{2n}(\varpi_t)$.
\end{cor}
\begin{proof} Consequence of  Proposition \ref{prop:symplectic1}.
\end{proof}

\begin{ex}For instance $ (2,3, 6,7 ,10,11,14,15)$  is a symplectic column consisting solely of intervals of length two $[ \s(x),x]$ with $x\in 2\Z$. This is an instance of the $\mathfrak{k}$-highest weight  symplectic columns in \cite{nsw}. More generally, the $\mathfrak{k}$-highest weight and  $\mathfrak{k}$-lowest weight symplectic tableaux are respectively given by the numbers
\begin{align}&u_k=2k-\frac{1+(-1)^k}{2}=\begin{cases}2k,&\text{ if } k \notin 2\Z,\\
2k-1,&\text{ if } k \in 2\Z,\label{numbers:u}
\end{cases}
\end{align}
and
\begin{align}
&v_k=2k-\frac{1+(-1)^{k+1}}{2}=\begin{cases}2k,&\text{ if } k \in 2\Z,\\
2k-1,&\text{ if } k \notin 2\Z. \label{numbers:v}
\end{cases}
\end{align}  whose first columns are in the conditions of Lemma \ref{lem:swapping}. See \cite[Section 5]{azreco}.
\end{ex}

Next corollary characterizes the symplectic columns $\bf a$ that decompose into maximal nonempty intervals   and satisfy $s({\bf a})=\bf a$.  That is,  we characterize  the maximal nonempty intervals  that appear in that decomposition.
\begin{cor} \label{cor:cups} Let $\mathbf{a}=\bigoplus_{i=1}^k {\bf A}_i \in SST_{2n}(\varpi_t)$  for some $k\ge 1$.
Then $\mathbf{a} \in SpT_{2n}(\varpi_t)$ if and only if, for $i=1,\dots,k$,
\begin{align}a_i\ge 2(t_1+\cdots+t_{i-1})+t_i+1.
\end{align}
\end{cor}

\begin{proof}  Recall the Definition \ref{maximalintervaldecompose}, for $i=1,\dots, k$, ${\bf A_i}=(a_i,a_i+1,\dots,a_i+t_i-1)\in SST_{2n}(\varpi_{t_i})$ with  $a_i\notin 2\Z$ and  $2\le t_i\in 2\Z$, such that $a_{i+1}-a_i\ge t_i+2$  and $t=t_1+\cdots+t_k\in  2\Z$.

"\emph{Only if part}." If ${\bf a}\in  SpT_{2n}(\varpi_t)$ then for $i=1,\dots,k$,
$$a_i+t_i-1\ge 2(t_1+\cdots+t_i)-1\Leftrightarrow a_i\ge 2(t_1+\cdots+t_{i-1})+t_i.$$
Since $a_i\notin 2\Z$ and $t_1,\dots,t_i\in 2\Z$, then for $i=1,\dots,k$,
$$a_i+t_i-1\ge 2(t_1+\cdots+t_i)-1\Leftrightarrow a_i\ge 2(t_1+\cdots+t_{i-1})+t_i\Leftrightarrow a_i\ge 2(t_1+\cdots+t_{i-1})+t_i+1.$$

"\emph{If} \emph{part}" Let $i\in \{1,\dots,k\}$ and $j\in\{1,\dots, t_i-1\}$. One has
$$a_i\ge 2(t_1+\cdots+t_{i-1})+t_i+1\Leftrightarrow a_i+t_i-1\ge 2(t_1+\cdots+t_i)-1$$
and by reverse induction assume $a_i+j\ge 2(t_1+\cdots+t_{i-1}+j+1)-1$ then
\begin{align}a_i+j-1\ge 2(t_1+\cdots+t_{i-1}+j+1)-1-1 = 2(t_1+\cdots+t_{i-1}+j)>2(t_1+\cdots+t_{i-1}+j)-1.
\end{align}
Therefore, $\bf a$ is symplectic.
\end{proof}
\begin{obs} Note ${\bf A}=(1,2)\in SST_{2n}(\varpi_2)$ is not symplectic because $a_1=1\ngeq t_1+1=3$ although $s({\bf A})=\bf A$.
\end{obs}

Finally, next corollary characterizes the maximal intervals of length $\ge 2$ in a symplectic column if they exist. Here we are not requiring that an interval length of $\ge 2$ is closed under the action of $\s$.

 \begin{cor}\label{cor:maximalinterval} Let $\mathbf{a} \in SST_{2n}(\varpi_t)$   such that the assertions of Lemma \ref{lem:swapping} do not hold. Then
 \begin{enumerate}
 \item
 If $\mathbf{a} \in SpT_{2n}(\varpi_t)$
 then there exists  a maximal  interval $I\subseteq \bf a$   with  $|I|\ge 2$ that does not satisfy point \eqref{intervaltwo} of that lemma.  In this case there exists a maximal interval  ${\bf A_1}=(a_1, a_1+1,\dots, a_1+t_1-1)$ for some $a_1\notin 2\Z$ and  $2\le t_1\in [t]\cap 2\Z$ such that $I$ is of  either form
 \begin{enumerate}
 \item $I=\bf A_1$, where  $a_1-1,~a_1+t_1\notin \bf a$, $\s({I})= I$  and $a_1\ge t_1+1$ or
 \item $I=(a_1-1){\bf A_1} $, where $\s(a_1-1)\notin\bf a$,  $a_1+t_1\notin \bf a$, and $a_1\ge t_1+2+1$ or
 \item  $I={\bf A_1}(a_1+t_1)$, where $a_1-1\notin \bf a$, $\s(a_1+t_1)\notin \bf a$ and $a_1\ge t_1+1$ or
 \item $I=(a_1-1){\bf A} (a_1+t_1)$, where  $\s(a_1-1), \s(a_1+t_1)\notin\bf a$ and $a_1\ge t_1+2+1$.
 \end{enumerate}

 \item If ${\bf a}$   either equals $I=\bf A_1$ where  $a_1\ge t_1+1$, or $I=(a_1-1){\bf A_1} $, where $a_1\ge t_1+2+1$, or $I={\bf A_1}(a_1+t_1)$, where  $a_1\ge t_1+1$ or $I=(a_1-1){\bf A} (a_1+t_1)$, where  $a_1\ge t_1+2+1$, for some $\bf A_1$ as in point $(1)$, then
  $I\in SpT_{2n}(\varpi_{|I|})$.
 \end{enumerate}
 \end{cor}

 \begin{proof} $(1)$ If $\mathbf{a} \in SpT_{2n}(\varpi_t)$ is not in the conditions of Lemma \ref{lem:swapping}  then there exists a maximal interval $I$ contained in $\bf a$ that does not satisfy condition  point \eqref{intervaltwo} of that lemma.  Therefore,  $I=[x,y]\subseteq \bf a$ where  $x<y$ and $x-1, y+1\notin\bf a$, is a maximal interval in $\bf a$, and by point \eqref{intervaltwo} in Lemma \ref{lem:swapping},  either $I=[x,x+1]=\{x,x+1\}$ with $x\notin 2\Z$ and $x-1, x+2\notin \bf a$, or $|I|\ge 3$. In the former case, $t=2$  and $I=\{a,a+1,\dots,a+t-1\}=\{a,a+1\}$ for some $a\notin 2\Z$. Since $a\notin 2\Z$, and $a=1\Rightarrow a+1=2\ngeq 2\times 2-1$,   and $\bf a$ is symplectic it forces that $a\ge 3=t+1$ with $t=2$.

 In the latter case, $|I|\ge 3$ and $I\subseteq\bf a$ is maximal. If $x$ is the least element of $I$ and $x\in2\Z$ then $s(x)=x-1\notin \bf a$ because  $I$ is maximal. If $y$ is the largest element of $I$ and $y\notin 2\Z$  then $s(y)=y+1\notin \bf a$. This means that in those cases $I\setminus \{x,y\}$ is an interval such that $s(I\setminus \{x,y\})=I\setminus \{x,y\}$ and  $I$ satisfies points $(b)$, or  $(c)$ or $(d)$. If the least and largest elements of $I$ are respectively odd and even then $I$ satisfies point $(a)$.

 $(2)$ If $I=(a_1-1){\bf A_1} $, where $a_1\ge t_1+2+1$, then
 $$a_1\ge t_1+2+1\Leftrightarrow a_1+t_1-1\ge t_1+2+1+t_1-1= 2(t_1+1)> 2(t_1+1)-1.$$

 If $I={\bf A_1}(a_1+t_1)$, where  $a_1\ge t_1+1$ then $$a_1\ge t_1+1\Leftrightarrow a_1+ t_1\ge 2t_1+1= 2(t_1+1)-1$$

If $I=(a_1-1){\bf A} (a_1+t_1)$ where  $a_1\ge t_1+2+1$, it remains to show that
$$a_1\ge t_1+2+1\Leftrightarrow a_1+t_1\ge t_1+2+1+t_1=2(t_1+2)-1.$$

  Therefore $I$ is symplectic.
 \end{proof}

Finally we characterize the symplectic columns via its decomposition by detaching the   subword that is by fixed the parity involution $\s$, this is
 the union of maximal nonempty intervals invariant  under the action of the parity involution $\s$,  and its remaining part  that consists of the subword that under the action of the parity involution is disjoint of it.

 \begin{thm}\label{thm:symplecticdecomposition}Let ${\bf a}\in SST_{2n}(\varpi_t)$. Then
  ${\bf a}\in SpT_{2n}(\varpi_t)$ if and only if, for some $k\ge 0$,
\begin{align}\label{generalsymplecticdecompose}
{\bf a}=
X^{(0)}{\bf A_1} X^{(1)}{\bf A_2}\cdots {\bf A_{k-1}}X^{(k-1)} {\bf A_k} X^{(k)},
\end{align}
where 
\begin{enumerate}
\item  for $i=0,1,\dots,k$, each $X^{(i)}$ is a column  word of length $p_i\ge 0$, such that

\begin{align}\s(X^{(0)}X^{(1)}\cdots X^{(k)})\cap {\bf a}=\emptyset,
\end{align}

\item for $ i=1,\dots,k $, each interval ${\bf A_i}=(a_i, a_i+1,\dots,a_i+t_i-1)$, $a_i\notin 2\Z$, is
a maximal nonempty interval  of even length $2\le t_i\in  2\Z$, that is, each $\bf A_i$ is a maximal nonempty interval such that $s(\bf A_i)=A_i$,   and
$$t=\sum_{i=1}^kt_i +\sum_{i=0}^k p_i  \in [n],$$
satisfying
\begin{align} \label{3}
&a_i\ge 2t^{(i-1)}+t_i+1, ~\mbox{ with $\displaystyle t^{(i-1)}:=\sum_{j=1}^{i-1}t_j+\sum_{j=0}^{i-1}p_j$, ~$t^{(0)}:=p_0$, and $t^{(k)}:=t$.}
\end{align}
\end{enumerate}
Moreover, this decomposition is unique in the sense that the maximal nonempty intervals ${\bf A_i}$ fixed by $\s$ are uniquely determined by $\bf a$, in the sense that their union constitute the part of $\bf a $ that is fixed by $\s$,  and what remains  is the subword ${\bf a} \setminus \bigoplus_{i=1}^k {\bf A_i}$ of $\bf a$ such that $\s({\bf a} \setminus \bigoplus_{i=1}^k {\bf A_i})\cap {\bf a}=\emptyset$.
\end{thm}
\begin{proof} "\emph{Only} \emph{if} \emph{part}."  Considering ${\bf a}\in SpT_{2n}(\varpi_t) $ as a set,
\begin{align}{\bf a}=\{x\in {\bf a}: \s(x)\in \bf a\}\sqcup \{x\in {\bf a}: \s(x)\notin \bf a\}.\label{decompositionbys}
\end{align}

Let $U:=\{x\in {\bf a}: \s(x)\in \bf a\}$ and $V=\{x\in {\bf a}: \s(x)\notin \bf a\}$. Then by definition of $V$, $\s(V)\cap{\bf a}=\emptyset$, and since $\s$ is an involution,  $\s(U)=U$. Therefore the sets $U$ e $V$ are uniquely defined by $\bf a$ and the decomposition \eqref{decompositionbys} is uniquely defined by $\bf a$.  By Lemma \ref{intervals}, if $U\neq \emptyset$,  $U=\bigoplus_{i=1}^k {\bf A}_i $ for some $k\ge 1$,  where each ${\bf A_i}=(a_i, a_i+1,\dots,a_i+t_i-1)$, with $a_i\notin 2\Z$ and $2\le t_i\in 2\Z$, and, by definition of $U$,  is a maximal nonempty interval of $\bf a$ such that $\s({\bf A_i})=\bf A_i$.

Henceforth,
decomposing ${\bf a}\in SpT_{2n}(\varpi_t)$ into maximal intervals of length $\ge 2$ and  detaching those $\bf A_i$  which are fixed by $\s$,
  we get  $U=\bigoplus_{i=1}^k {\bf A}_i $ for some $k\ge 0$, and
 $$V={\bf a} \setminus \bigoplus_{i=1}^k {\bf A}_i.$$

Thus,  \eqref{generalsymplecticdecompose} is the decomposition that we get for $\bf a$ as a column word. Since $\bf a$ is symplectic the intervals $\bf A_i$ are forced to satisfy conditions \eqref{3}, and by definition of $V$, the column word defined by $V$ satisfy the conditions of Lemma \ref{lem:swapping}.

"\emph{If} \emph{part}." By induction on $k$. For $k=0$, one has ${\bf a}=X_0$ in the conditions of Lemma \ref{lem:swapping}, hence, by Corollary \ref{cor:withoutA},  ${\bf a}=X^{(0)}$ is symplectic.
Let $k\ge 1$ and assume the assertion true for $k-1$. Let
\begin{align}{\bf a}=
X^{(0)}{\bf A_1} X^{(1)}{\bf A_2}\cdots {\bf A_{k-1}}X^{(k-1)} {\bf A_k} y_{1}\cdots y_{p_k}, \quad X^{(k)}=y_{1}\cdots y_{p_k}.
\end{align}
We have to prove that the piece ${\bf A_k} y_{1}\cdots y_{p_k}$ as  a suffix  of $\bf a$ is symplectic.
From point \eqref{3}, \begin{align*}& a_k\ge 2(t_1+\cdots+t_{k-1})+ 2(p_0+p_1+\cdots+p_{k-1})+t_k+1\\
&\Leftrightarrow\\
 &a_k+t_{k}-1\ge 2(t_1+\cdots+t_{k-1})+ 2(p_0+p_1+\cdots+p_{k-1})+t_k+1+t_k-1=\\
 &=2(t_1+\cdots+t_{k-1}+t_k)+ 2(p_0+p_1+\cdots+p_{k-1})\\
& >2(t_1+\cdots+t_{k-1}+t_k)+ 2(p_0+p_1+\cdots+p_{k-1})-1.
 \end{align*}

On the other hand $\s(y_j)\notin \bf a$ for $j=1,\dots,p_k$.

If $y_1\in 2\Z$, since $a_k+t_k-1\notin 2\Z$, then
\begin{align*}&y_1\ge a_k+t_k-1+2\ge 2\sum_{i=1}^{k-1}t_i+2\sum_{i=0}^{k-1}p_i+t_k+1+t_k+1, \mbox{ by point \eqref{3}}\\
&=2\sum_{i=1}^{k-1}t_i+2(\sum_{i=0}^{k-1}p_i+1)>2(\sum_{i=1}^{k-1}t_i+\sum_{i=0}^{k-1}p_i+1)-1
\end{align*}

If $y_1\notin 2\Z$ and $y_1=a_k+t_k-1+1$, then
\begin{align*}&y_1\ge a_k+t_k\ge 2\sum_{i=1}^{k-1}t_i+2\sum_{i=0}^{k-1}p_i+t_k+1+t_k, \mbox{ by point \eqref{3}}\\
&=2\sum_{i=1}^{k-1}t_i+2(\sum_{i=0}^{k-1}p_i+1)-1\\
&=2(\sum_{i=1}^{k-1}t_i+\sum_{i=0}^{k-1}p_i+1)-1.
\end{align*}

Write $p:=p_k\ge1$. Let $p>1$ and by induction assume
\begin{align}\label{induction}y_j\ge 2(\sum_{i=1}^{k-1}t_i+\sum_{i=0}^{k-1}p_i+j)-1, \mbox{ for $j=1,\dots, p-1$.}
\end{align}

 Then since $\s(y_{j})\notin \bf a$ for any $j\ge 1$, and similarly to the proof of Corollary \ref{cor:withoutA}, one has three cases:
\begin{itemize}
\item $y_p\in 2\Z\Rightarrow s(y_p)=y_p-1\notin \bf a$  and  $y_p\ge y_{p-1}+2$. Then by induction hypothesis,
$$y_p\ge y_{p-1}+2\ge 2(\sum_{i=1}^{k-1}t_i+\sum_{i=0}^{k-1}p_i+p-1)-1+2=2(\sum_{i=1}^{k-1}t_i+\sum_{i=0}^{k-1}p_i+p)-1.$$

\item $y_p\notin 2\Z$ and  $y_{p-1}\notin 2\Z\Rightarrow \s(y_{p-1})=y_{p-1}+1\notin \bf a$. Then $y_p\ge y_{p-1}+2$ and  by induction hypothesis $$y_p\ge  y_{p-1}+2\ge 2(\sum_{i=1}^{k-1}t_i+\sum_{i=0}^{k-1}p_i+p-1)-1+2=2(\sum_{i=1}^{k-1}t_i+\sum_{i=0}^{k-1}p_i+p)-1$$

    \item  $y_p\notin 2\Z$ and  $y_{p-1}\in 2\Z$. If $y_p=y_{p-1}+1$ and $p=2$ then ${\bf a}=y_1(y_1+1)$ and $2\le y_1\in 2\Z$. In this case $y_1+1\ge 3=2\times 2-1$.

        If $y_p=y_{p-1}+1$ and $p\ge 3$ then since $\s(y_{p-1})=y_{
        p-1}-1 \notin \bf a$ one has
        \begin{align}{\bf a}=\cdots y_{p-2}y_{p-1}(y_{p-1}+1) \mbox{ and $y_{p-2}\le y_{p-1}-2$}
        \end{align}
        Therefore $y_{p-1}\ge y_{p-2}+2$ and $y_p=y_{p-1}+1\ge y_{p-2}+3$. By induction hypothesis \eqref{induction} and since $y_p\notin 2\Z$
        $$y_p\ge y_{p-2}+3\ge 2(\sum_{i=1}^{k-1}t_i+\sum_{i=0}^{k-1}p_i+p-2)-1+3=2(\sum_{i=1}^{k-1}t_i+\sum_{i=0}^{k-1}p_i+p)-2\in 2\Z.$$
        Henceforth
        $$y_p\ge 2(\sum_{i=1}^{k-1}t_i+\sum_{i=0}^{k-1}p_i+p)-2+1=2(\sum_{i=1}^{k-1}t_i+\sum_{i=0}^{k-1}p_i+p)-1.$$
\end{itemize}
Therefore $\bf a$ is symplectic.
\end{proof}

\section{The  inverse of the reduction map}\label{sec:inverseredu}
\subsection{ The reduction map}   For the reader convenience this section recalls several properties of removal and reduction maps in \cite{watanabe} and \cite{azreco}.
We fix $l \in [0, 2n]$ and $\mathbf{a} = (a_1, \dots , a_l)$ a column  in $SST_{2n}(\varpi_l)$.
Often  $\mathbf{a}$ is regarded as a set.
The \emph{removal} \emph{subword }of $\mathbf{a}$ \cite{watanabe} is defined to be the subword $\mathrm{rem}(\mathbf{a})$ of  $\mathbf{a}$ obtained by the following recursive formula:

\begin{align}\label{removal}
\mathrm{rem}(\mathbf{a}) :=
\begin{cases}
\emptyset,& \mbox{ if } l\le 1,\\
\mathrm{rem}(a_l,\dots, a_{l-2})  (a_{l-1},a_l),& \mbox{ if } l\ge 2, a_l\in   2\mathbb{Z}, a_{l-1} = a_l-1,\\
 &\mbox{ and }
a_l < 2l -|\rem(a_1,\dots,a_{l-2})|-1,\\
\rem(a_1,\dots,a_{l-1})& \mbox{ otherwise}.
\end{cases}
\end{align}

\begin{defi}\cite{watanabe} \label{def:rem} For the column
$\mathbf{a}$, the new column $\mathrm{red}(\mathbf{a})$, reduction of $\mathbf{a}$, is defined to be the one obtained from $\mathbf{a}$ by removing the entries in
the set $\mathrm{rem}(\mathbf{a})$,
$$\redu( \mathbf{a})=\mathbf{a}\setminus \rem(\mathbf{a})\in SST_{2n}(\varpi_{l-|\rem(\mathbf{a})|}).$$
\end{defi}
In fact  Proposition \ref{prop:redu} \cite{watanabe} shows that $\redu( \mathbf{a})\in SpT_{2n}(\varpi_{l-|\rem(\mathbf{a})|})$, where $l-|\rem(\mathbf{a})|$, satisfy further conditions.

We recall  useful properties of removable entries in $\mathbf{a}$. In particular, we highlight the parity swapping involution $\s$ with a meaningful role in the removal of pairs and consequently in the reduction map and  in our computation of the inverse reduction map $\redu$.

\begin{prop}\label{prop:removcontain}\cite[Proposition 2.5.3]{watanabe}
\begin{enumerate}
\item[(1)] If $ a_l \in \rem(\mbf{a})$, then $a_l\in 2\mathbb{Z}$.

\item[(2)] If $a_l\notin \rem(\mbf{a})$ then $ \rem(\mbf{a})= rem(a_1,\dots, a_{l-1})$.

\item[(3)] If $a_l$ is odd then $ a_l\notin \rem(a)$, and $\rem(\mbf{a})= \rem(a_1,\dots, a_{l-1})$.

\item[(4)] If $a_l \in \rem(\mbf{a})$, then $a_l\in 2\mathbb{Z}$ and $\rem(a_1,\dots, a_{l-1})= \rem(a_1,\dots, a_{l-2})$.

\item [(5)] $\rem(a_1,\dots, a_k) \subset \rem(\mbf{a})$ for all $k \in[0,l]$.

\end{enumerate}
\end{prop}
In the next proposition points $(2)$ and $(3)$ together  are an alternative to the usage of Definition \ref{def:rem} to compute $\rem(\bf a)$  allowing its computation easily.

\begin{prop}\cite[Proposition 4.2.2, 4.2.7, Corollary 4.2.8]{watanabe}\label{prop:rem}
\begin{enumerate}

\item[(1)] If $i\in [1,l]$ and $a_j\notin \rem(\mbf{a})$, for $j\in[i,l]$ then $\rem(\mbf{a})=\rem(a_1,\dots,a_{i-1})$.

\item[(2)] for each $i\in [1,l]$, $a_i\in \rem(\mbf{a})$ if and only if one of the following holds
\begin{enumerate}
\item[(a)] $a_i$ odd, $i<l$, $a_{i+1}=a_i+1$ and $a_i<2i-|\rem(a_1,\dots,a_{i-1})|$

\item[(b)] $a_i$ even, $i>1$, $a_{i}=a_{i-1}+1$ and $a_i<2i-|\rem(a_1,\dots,a_{i-2})|-1$
\end{enumerate}

\item [(3)]  $a_i \in \rem(\mbf{a})$ if and only if $\s(a_i) \in \rem(\mbf{a})$.
Consequently, $|\rem(\mbf{a})|\in 2\mathbb{Z}$.



\end{enumerate}

\end{prop}

\begin{cor} \label{empty0} For $l \in [0, 2n]$,

\begin{enumerate}
\item[(1)] $l$ even $\Rightarrow \rem(1,2,\dots,l)=(1,2,\dots,l)=\redu(1,2,\dots,l)=\emptyset$.

\item[(2)] $l$ odd $\Rightarrow \rem(1,2,\dots,l)=\rem(1,2,\dots,l-1)= (1,2,\dots,l-1)\Rightarrow \redu(1,2,\dots,l)=\{l\}$.

\item[(3)] $\rem(\mbf{a})=\mbf{a}$ if and only if $l$ is even and $\mbf{a}=(1,2,\dots,l)$.

\item[(4)] $\redu(\mathbf{a})=\emptyset$ if and only if $l$ is even and $\mathbf{a}=(1,2,\dots,l)$.

\end{enumerate}
\end{cor}

\begin{ex}  

$\rem(6,7,8,9,10,11,12)=(11,12)$, $\redu(6,7,8,9,10,11,12)=(6,7,8,9,10)$

$\rem(1,2,3,4,5)=(1,2,3,4)$, $\redu(1,2,3,4,5)=(5)$.
\end{ex}

\begin{prop} \label{prop:redu}\cite{watanabe,nsw} Let $l\in [0,2n]$ and  $\mathbf{a}=(a_1,\dots,a_l) \in SST_{2n}(\varpi_l)$.
Then
\begin{enumerate}

\item[(1)] $\redu(\mathbf{a})$ is symplectic.

    \item[(2)] $\redu(\mathbf{a})=\mathbf{a}$ if and only if $\mathbf{a}$ is symplectic.

\item[(3)] \cite[Proposition 4.3.6]{watanabe} The reduction map $\redu$ on $SST_{2n}(\varpi_l)$ is the injective assignment

\begin{align}\redu:SST_{2n}(\varpi_l)&\rightarrow\bigsqcup_{\begin{smallmatrix} 0\le t\le min\{l,2n-l\}\\
l-t\in 2\mathbb{Z}
\end{smallmatrix}}SpT_{2n}(\varpi_t)\nonumber\\
\mathbf{a}&\mapsto \redu( \mathbf{a})=\mathbf{a}\setminus \rem(\mathbf{a}).
\label{redumap}
\end{align}
\end{enumerate}
\end{prop}


We now show that the reduction map \eqref{redumap} is also surjective. For a fixed $n\in\mathbb{N}$, let $l\in [0,2n]$ and $t\in[0,n]$ such that  $0\le t\le min\{l,2n-l\}$ and $l-t\in 2\Z$. Then we define the expanding map

\begin{align}\label{exp}
\rm{exp}_t:SpT_{2n}(\varpi_t)&\rightarrow SST_{2n}(\varpi_l)\quad\\
\mathbf{a}&\mapsto\rm{exp}(\mathbf{a}):=T,\nonumber
\end{align}
such that $\redu(\rm{exp}(\bf a))=a$. That is, $\rm{red}_{t}^{-1}:=\redu^{-1}_{|SpT_{2n}(\varpi_t)}=\rm{exp}_t$. The remaining sections are devoted to provide procedures to explicitly define $\redu^{-1}(\mathbf{a})$ for a symplectic column $\mathbf{a}\in SpT_{2n}(\varpi_t)$.

\subsection{The procedures to expand symplectic columns}
  The procedure to expand  symplectic columns in the conditions of Lemma  \ref{lem:swapping} is given now. These are the symplectic columns in the conditions of Theorem \ref{thm:symplecticdecomposition} when $k=0$, that is,  the symplectic columns without nonempty  intervals fixed by the parity involution $\s$. (Columns with this property are guaranteed to be symplectic by Corollary \ref{cor:withoutA}.) In other words these are the symplectic columns $\bf a$ such that $${\bf a}\cap \s({\bf a})=\emptyset.$$

The next theorem was stated in \cite{azreco} but a complete and constructive proof is given here.

\begin{thm}\label{thm:noconsecutive} Let $l\in [0,2n]$ and $t\in[0,n]$ such that  $0\le t\le min\{l,2n-l\}$ and $l-t\in 2\Z$. Let $\mathbf{a} = (a_1, \dots , a_t)\in SpT_{2n}(\varpi_t)$ such that, for every $1\le i\le t$, $\s(a_i)\notin \bf{a}$. Then,

\begin{align}\label{redu:nofactors}
\rm{red}_{t}^{-1}:=\redu^{-1}_{|SpT_{2n}(\varpi_t)}:SpT_{2n}(\varpi_t)&\rightarrow SST_{2n}(\varpi_l)\quad\\
\mathbf{a}&\mapsto\rm{red}_{t}^{-1}(\mathbf{a}):=T_0(a_1)T_1\cdots (a_t) T_t,\nonumber
\end{align}
where \begin{align}\label{inverseredu:nofactors}
&T_0=(1,2,\dots l_1),\\
&({a}_i)T_i=
\begin{cases}
({a}_i) ( a_i+1,\dots, a_i+l_{i+1}),
  &
 \mbox{if } a_i\in 2\mathbb{Z}\\
({a}_i) ( a_i+1+1,a_i+1+2,\dots, a_i+1+l_{i+1}),
 &\mbox{if } a_i\notin 2\mathbb{Z}\\
\end{cases},& 1\le i\le t, \label{inverseredu:nofactors2}
\end{align}
and
$l_1,\dots, l_t$, $l_{t+1}$ are nonnegative even numbers defined by
\begin{align}\label{l:nofactors}
&l_1=
\begin{cases} min\{a_1-2, l-t\},& \mbox{if } a_1\in 2\mathbb{Z}\\
min\{ a_1-1, l-t\},&\mbox{if } a_1\notin 2\mathbb{Z},\\
\end{cases}
\end{align}
\begin{align}
&l_{i+1}=
\begin{cases}min\{a_{i+1}-a_i-2, l-t-\sum_{k=1}^i l_k\}, &\mbox{if }
a_i, a_{i+1}\in 2\mathbb{Z} \mbox{ or } a_i, a_{i+1}\notin 2\mathbb{Z}
\\
min\{{ a_{i+1}-a_i-1},~~ l-t-\sum_{k=1}^i l_k\},&\mbox{if }
a_i\in 2\mathbb{Z}, a_{i+1}\notin 2\mathbb{Z} \\
min\{a_{i+1}-a_i-3,  l-t-\sum_{k=1}^i l_k\},&\mbox{if }
 a_i\notin 2\mathbb{Z}, ~~~ a_{i+1}\in 2\mathbb{Z}
\end{cases},\quad \mbox{ for } 1\le i<t,\label{ells}
\end{align}
and
$$l_{t+1}=l-t-\sum_{i=1}^{t}l_i.$$
\end{thm}
\begin{proof}
 Recall Lemma \ref{lem:swapping}. We start by checking that $l_i$ is a non negative integer for $i=1,\dots,t,t+1$.
By definition $l-t\ge 0$ and by \eqref{l:nofactors}, $l-t\ge l_1\ge 0$. Hence $l-t-l_1\ge 0$ and by \eqref{ells} $l-t-l_1\ge l_2\ge 0$. Hence $l-t-l_1-l_2\ge 0$ and $l-t-l_1-l_2\ge l_3\ge 0$ by \eqref{ells}. Continuing in this manner, by  \eqref{l:nofactors} and \eqref{ells}
  \begin{align*}  l-(i+\sum_{k=1}^i l_k)-(t-i)=l-\sum_{k=1}^i l_k-i-t+i=l-t-\sum_{k=1}^i l_k\ge l_{i+1}\ge 0, \mbox{ for $0\le i<t$.}
  \end{align*}
  Therefore, $$l_{t+1}=l-t-\sum_{i=1}^{t}l_i\ge 0$$ is well defined and
  indeed $T_0({a}_1)T_1\cdots({a}_t) T_t$ 
  has length
\begin{align*}\ell(T_0({a}_1)T_1\cdots ({a}_t)T_{t})=\sum_{i=0}^{t}\ell(T_i)+t=\sum_{i=1}^{t}l_i+l_{t+1}+t
=l.
\end{align*}

\bigskip

Note ${\bf a}\sqcup \s({\bf a})\subseteq [a_t]$, if $a_t\in 2\Z$, and ${\bf a}\sqcup s({\bf a})\subset [a_t+1]$, otherwise, and $|{\bf a}\sqcup s({\bf a})|=2t$.  Therefore, $a_t\ge 2t-1$ if $a_t\notin 2\Z$ and  $a_t\ge 2t$  if $a_t\in 2\Z$,  consistent with  that $\bf a$ is symplectic.

We need by Proposition \ref{prop:rem} to add  $\frac{l-t}{2}$ pairs ($l-t\in 2\Z$) of consecutive entries $\{x,\s(x)\}$ in $[1,2n]$ to the column $a=(a_1,\dots, a_t)\in SpT_{2n}(\varpi_t)$ to get a column $T$ in $SST_{2n}(\varpi_l)$ such that $\{x,\s(x)\}\cap \big({\bf a}\cup s({\bf a})\big)=\emptyset$;  and, in addition, these $\frac{l-t}{2}$ pairs define $\rem(T)$, that is, $\redu(T)=T\setminus \rem(T)=a$. We add these pairs of entries to $\bf a$ as earliest as possible as follows.

Let $W:=[2n]\setminus \left({\bf a}\sqcup \s({\bf a})\right)$ where $|W|=2n-2t$. Since $ 2n-l\ge t\Leftrightarrow |W|=2n-2t\ge l-t$, it follows $|W|\ge l-t$.
Let $W^l_t$ be the subset of $W$ consisting of the first $l-t$ entries of $W$. Consider, for $i=1,\dots,t+1$,

 \begin{align}W_{i-1}:=W^l_t\cap[a_{i-1},a_i],  \mbox{   with $a_0:=0$ and $a_{t+1}:=+\infty$}.
 \end{align}
  That is, $W^l_t=W_0\sqcup W_1\sqcup\cdots\sqcup W_{t-1}\sqcup W_t$,
and, recalling Lemma \ref{lem:swapping}, for $i=1,\dots,t+1$, one gets \eqref{inverseredu:nofactors}, \eqref{inverseredu:nofactors2}, \eqref{l:nofactors}, \eqref{ells},
\begin{align}W_{i-1}=(a_{i-1})T_{i-1}\setminus (a_{i-1}),\quad |W_{i-1}|=l_{i}\in 2\Z,\end{align}  where in particular

\begin{align}W_0=\{1,\dots,l_1\}=T_0, \mbox{  and } W_{t}=W^l_t\cap[a_{t},+\infty]=(a_t)T_t\setminus (a_t).
\end{align}

It incurs that each $W_i$, if not empty, is a sequence of consecutive integers starting with an odd number and terminating with an even number. (If for some $i$, $W_{i-1}$  starts with an even number $x$ then $\s(x)=a_{i-1}$ which contradicts the definition $W:=[2n]\setminus \left({\bf a}\sqcup \s({\bf a})\right)$. Similarly if the terminating entry of $W_{i-1}$ is an odd number $x$ then $\s(x)=a_i$.

Let $T:=T_0(a_1)T_1\cdots(a_t) T_t$.
Observe  that
when  it holds, for some $i\in[0,t]$,
$l_{i+1}=l-t-\sum_{k=1}^il_k\Leftrightarrow l=t+\sum_{k=1}^{i+1}l_k$,   in \eqref{l:nofactors} or  \eqref{ells}, it follows that  $W_{j}=()$ for all $j\ge i+1$, and, in particular, $W_t=\emptyset$.
Hence,
 $T$ is such that
\begin{align}\label{construction}(W^l_t\cap[a_{0},a_1])\bigsqcup\bigg(\bigsqcup_{i=2}^{t-1}(W^l_t\cap[a_{i-1},a_i])\bigg)\bigsqcup \big(W^l_t\cap[a_{t-1},a_t]\big)\bigsqcup \big(W^l_t\cap[a_t,+\infty]\big)=T\setminus {\bf a}.
\end{align}

We now analyze the construction of $T\in SST_{2n}(\varpi_l)$ from ${\bf a}\in SpT_{2n}(\varpi_{t})$ in  \eqref{construction} (or \eqref{redu:nofactors})  when we pass to ${\bf a'}=(a'_1,\dots,a'_{t-1})={\bf a}\setminus \{a_t\}\in SpT_{2n}(\varpi_{t-1})$ to be expanded in $SST_{2n}(\varpi_{l-1})$.
  Note that for $t\ge 1$, under the hypothesis conditions on $t$,  one also has $ 0\le t-1\le min\{l-1, 2n-(l-1)\}$, and $l-t=(l-1)-(t-1)\in 2\Z$.

For ${\bf a'}=(a'_1,\dots,a'_{t-1})={\bf a}\setminus \{a_t\}$,
it holds $\s({\bf a'})=\s({\bf a})\setminus (\s(a_t))$ and
\begin{align}{\bf a}\sqcup \s({\bf a})&={\bf a'}\sqcup \s({\bf a'})\sqcup \{a_t,\s(a_t)\}\\
&=\{a_1,\s(a_1)\}\sqcup\{a_2,\s(a_2)\}\sqcup\cdots\sqcup\{a_{t-1},\s(a_{t-1})\}\sqcup \{a_t,\s(a_t)\}
\end{align}
where, by definition of the involution $\s$,  each set $\{a_i,\s(a_i)\}$ consists of pairs of consecutive integers.
Let
\begin{align*}&W':=[2n]\setminus\left({{\bf {a'}}}\cup \s({\bf {a'}})\right)=W\sqcup \{a_t,\s(a_t)\}.
\end{align*}
 The set $W$ consists of  all integers of $[2n]$ except those in $\bf a\sqcup \s(a)$ and $|W|=2(n-t)$ is even. This implies that
 $$W=Z_0\sqcup Z\sqcup Z_1$$
 with

 $$Z_0=[1,min(a_1,\s(a_1))-1] ~\mbox{ and }  Z_1=[max(a_t,\s(a_t))+1,2n],$$
  and

   $$Z= [max(a_1,\s(a_1))+1, min(a_t,\s(a_t))-1]\setminus\bigsqcup_{i=2}^{t}\{a_i,\s(a_i)\}.$$
 of even cardinalities by \eqref{minmax}.
 Also

 $$  W'=Z_0\sqcup Z'\sqcup Z'_1,$$  where $ max(a_{t-1},\s(a_{t-1}))+1 \le min(a_t,\s(a_t))\le 2n$,
  and

  \begin{align}Z'_1&=[max(a_{t-1},\s(a_{t-1}))+1,2n]=\nonumber\\
  &=[max(a_{t-1},\s(a_{t-1}))+1,max(a_t,\s(a_t))]\sqcup [max(a_t,\s(a_t))+1 ,2n]\nonumber\\
  &=[max(a_{t-1},\s(a_{t-1}))+1,max(a_t,\s(a_t))]\sqcup Z_1,
  \end{align}

  $$Z'= [max(a_1,\s(a_1))+1, min(a_{t-1},\s(a_{t-1}))-1]\setminus\bigsqcup_{i=2}^{t-1}\{a_i,\s(a_i)\}.$$
 of even cardinality by \eqref{minmax}.
 Hence $max(a_t,\s(a_t))+1\in W,W'.$

  Since $l-t\in 2\Z$ and the cardinalities of $Z_0$, $Z$ and $Z_1$ are even then $2\le |W^l_t\cap Z_1|\in 2\Z$ if $W_t\neq \emptyset$. Note $\{a_t,\s(a_t)\}=[min(a_t,\s(a_t)), max(a_t,\s(a_t)) ]$,
 $[a_t,\s(a_t),max(a_t \s(a_t))+1,max(a_t \s(a_t))+2]$ is an interval  and
 $$[max(a_t \s(a_t))+1,max(a_t \s(a_t))+2]= W^l_t\cap[a_t,\s(a_t),max(a_t \s(a_t))+1,max(a_t \s(a_t))+2].$$
Thus $$W^l_t=W^l_t|_{<min(a_t, \s(a_t))}\sqcup W_t,$$
where $W_t=W_t^l|_{>max(a_t,\s(a_t))}=W|_{>max(a_t,\s(a_t))}\cap Z_1$ is an interval and  $|W_t^l|_{>max(a_t,\s(a_t))}|\ge 2$.

Define ${W'}^{l-1}_{t-1}$  to be the set  of the first $l-t=(l-1)-(t-1) $ elements of $W'$ obtained from $W_t^l$ by adjoining $\{a_t,\s(a_t)\}$ and omitting  the two last  consecutive integers if $W_t\neq\emptyset$.

 Therefore,
 \begin{align}&{W'}^{l-1}_{t-1}=\nonumber\\
&=W^l_t|_{<min(a_t, \s(a_t)}\sqcup\{a_t,\s(a_t)\}\sqcup (W_t\setminus \{\mbox{ the two last consecutive integers of $W_t^l$}\})\nonumber\\
&=W_0\sqcup\cdots\sqcup W_{t-2}\sqcup (W_{t-1}\sqcup\{a_t,\s(a_t)\}\sqcup W_t^-).
\end{align}
where $W_t^-$ is $W_t$ minus its two last elements, equivalently,
$$W_t^-=W_t\setminus \{\mbox{ the two last consecutive integers of $W_t^l$}\}.$$
Define
\begin{align}W'_{t-1}:=W_{t-1}\sqcup\{a_t,\s(a_t)\}\sqcup W_t^-.\label{wminus}
\end{align}
Henceforth
\begin{align}&{W'}^{l-1}_{t-1}=W_0\sqcup\cdots \sqcup W_{t-2}\sqcup W'_{t-1},\\
& |W_{i-1}|=l_i,\quad 1\le i\le t-1,\quad  |W'_{t-1}|=l_t+2+l_{t+1}-2=l_t+l_{t+1}:=l'_t
\end{align}
and ${W'}^{l-1}_{t-1}\cap [a_{i-1},a_i]=W_{i-1}=( a_i)T_{i-1}\setminus\{a_{i-1}\}$, for $i=1,\dots,t-1$, with $a_0=0$,
while $${W'}^{l-1}_{t-1}\cap [a_{t-1},+\infty)= W'_{t-1}$$ and
$$( a_{t-1})T'_{t-1}=(a_{t-1})W'_{t-1}, \mbox{ with } (a_{t-1})T'_{t-1}=(a_{t-1})W_{t-1}\sqcup\{a_t,\s(a_t)\} \mbox{ if } W_t^-=\emptyset.$$


Hence, one defines
 $T':=T_0( a_1)T_1\cdots (a_{t-2})T_{t-2}( a_{t-1})T'_{t-1}$  with 
 $$\ell (T')=l_1+\cdots+l_{t-1}+l_t+l_{t+1}+t-1=l-1$$ such that
\begin{align}&T'\setminus {\bf a'}=\bigsqcup_{i=1}^{t-1}\big(W'^{l-1}_{t-1}\cap[a_{i-1},a_i]\big)\bigsqcup \big(W'^{l-1}_{t-1}\cap[a_{t-1},+\infty)\big)\nonumber\\
&=\bigsqcup_{i=1}^{t-1}\big(W^{l}_{t}\cap[a_{i-1},a_i])\bigsqcup \big(W^{l}_{t}\cap[a_{t-1},a_t])\bigsqcup\{a_t,\s(a_t)\}\sqcup W_t^-\big).
\end{align}

\bigskip


We now show by induction on $t\ge 0$ that $\redu(T)=\bf a$.

 If $t=0$, $\bf a=()$ and $s(\bf a)=()$, $W^l_0\cap [0,+\infty)=W^l_0=T=(1,\dots,l)$ and $l_1=l\in 2\Z$. Clearly $\redu(1,\dots,l)=()$.

Let $t\ge 1$ and
by induction assume the assertion to be true for $t-1$. One has then by induction that
\begin{align}{\bf a'}=(a_1,\dots,a_{t-1})&=\redu(T_0( a_1)T_1\cdots( a_{t-2}) T_{t-2}T'_{t-1})\nonumber\\
&=\redu(T_0( a_1)T_1\cdots ( a_{t-2})T_{t-2}(a_{t-1})W'_{t-1})\\
&= \redu(T_0( a_1)T_1\cdots ( a_{t-2})T_{t-2}(a_{t-1})W_{t-1}\sqcup\{a_t,\s(a_t)\}\sqcup W_t^-).
\end{align}

where \begin{align}\label{formulainduction}&T_0( a_{1})T_1\cdots( a_{t-2}) T_{t-2}T'_{t-1}=\nonumber\\
&=\begin{cases}
T_0( a_{1})T_1\cdots( a_{t-2}) T_{t-2}(a_{t-1})W_{t-1}({a}_{t}).(\s(a_t)). ( \s(a_t)+1,\s(a_t)+2,\dots, \s(a_t)+l_{t+1}-2),&
 \mbox{ if } a_t\notin 2\mathbb{Z},\\
 T_0( a_{1})T_1\cdots ( a_{t-2})T_{t-2}(a_{t-1})W_{t-1}.(\s(a_{t}))({a}_{t})(a_t+1,a_t+2,\dots, a_{t}+l_{t+1}-2)),&
 \mbox{ if } a_t\in 2\mathbb{Z}\\
 \end{cases}
 \end{align}

It follows that in particular
\begin{align}\label{reduinduction}\redu(T_0( a_{1})T_1\cdots( a_{t-2}) T_{t-2}(a_{t-1})W_{t-1})=(a_1,\dots,a_{t-1})\\
=\redu(T_0( a_{1})T_1\cdots ( a_{t-2})T_{t-2}(a_{t-1})W_{t-1}\sqcup\{a_t,\s(a_t)\}).
\end{align}

Now,
\begin{align}\label{reduall}&T_0( a_{})T_1\cdots ( a_{t-1})T_{t-1}( a_{t})T_{t}=\nonumber\\
&= \begin{cases}
T_0( a_{1})T_1\cdots( a_{t-2}) T_{t-2}(a_{t-1})W_{t-1}({a}_{t}) ( \s(a_t)+1,s(a_t)+2,\dots, \s(a_t)+l_{t+1}),&
 \mbox{ if } a_t\notin 2\mathbb{Z},\\
 T_0( a_{1})T_1\cdots ( a_{t-2})T_{t-2}(a_{t-1})W_{t-1}({a}_{t})(a_t+1,a_t+2,\dots, a_{t}+l_{t+1}),&
 \mbox{ if } a_t\in 2\mathbb{Z}\\
 \end{cases}
 \end{align}
where $l_{t+1}=(l-t)-\sum_{i=1}^tl_i.$
 Therefore  from \eqref{reduinduction}, and because $\s(a_t)\notin W_{t-1}$ (by definition of $W$) and consequently $a_t$ is not paired with $\s(a_t)$ in \eqref{reduall},
$$\redu( T_0( a_{1})T_1\cdots ( a_{t-2})T_{t-2}(a_{t-1})W_{t-1}({a}_{t}))=\bf a$$

It remains to analyse in \eqref{reduall}. $( \s(a_t)+1,\s(a_t)+2,\dots, \s(a_t)+l_{t+1})$ for $a_t\notin 2\mathbb{Z}$, and $(a_t+1,a_t+2,\dots, a_{t}+l_{t+1})$ for $a_t\in 2\mathbb{Z}$. For for $a_t\notin 2\mathbb{Z}$, and  $1\le j\le l_{t+1}$ odd, we have to show

$$\s(a_t)+j<2(l-t+t+j)-|\rem(T_0( a_{1})T_1\cdots ( a_{t-2})T_{t-2}(a_{t-1})W_{t-1}({a}_{t}) ( \s(a_t)+1,\dots,\s(a_t)+j-1)|$$
$$\Leftrightarrow$$
$$\s(a_t)+j<2(l-t+t+j)-(l-t+j-1)$$
$$\Leftrightarrow$$
\begin{align}\s(a_t)+j<2(l+j)-(l-t+j-1)\label{cat}\end{align}

From \eqref{formulainduction}, if $a_t\notin2\Z$, for all $1\le j\le l_{t+1}-2$, and $j$ odd,
\begin{align}&\rem(T_0( a_{1})T_1\cdots ( a_{t-2})T_{t-2}(a_{t-1})W_{t-1}({a}_{t}).(\s(a_t)). ( \s(a_t)+1,\s(a_t)+2,\dots, \s(a_t)+j,\s(a_t)+j+1))\\
&=W_1\cdots W_{t-1}({a}_{t}).(\s(a_t)).\s(a_t)+1,\s(a_t)+2,\dots, \s(a_t)+j,\s(a_t)+j+1)
\end{align}
which means, by $\sum_{i=1}^t l_i=l-t$,

\begin{align}&\s(a_t)+j<2(l-t+t+1+j)-\nonumber\\
&-|\rem(T_0( a_{1})T_1\cdots ( a_{t-2})T_{t-2}(a_{t-1})W_{t-1}({a}_{t}).(\s(a_t)). ( \s(a_t)+1,\s(a_t)+2,\dots, \s(a_t)+j-1)|\nonumber\\
&=2(l-t+t+j)+2-(l-t+j-1)-2\nonumber\\
&=2(l+j)-(l-t+j-1) \mbox{ which confirms \eqref{cat}}
\end{align}
 In particular, for $j=l_{t+1}-3$, one has
 \begin{align*}\s(a_t)+l_{t+1}-3&<2(l-t+t+1+l_{t+1}-3)-(l-t+2+l_{t+1}-4)
=2(l+l_{t+1}-2)-(l-t+l_{t+1}-2)\\
&\Leftrightarrow\\
&\s(a_t)+l_{t+1}-1<2(l+l_{t+1}-2)-(l-t+l_{t+1}-2)+2\\
&\s(a_t)+l_{t+1}-1<2(l+l_{t+1}-2+1)-(l-t+l_{t+1}-2)
\end{align*}
which gives \eqref{cat} for $j=l_{t+1}-1$.

For for $a_t\in 2\mathbb{Z}$, and  $1\le j\le l_{t+1}$ odd, we have to show

$$a_t+j<2(l-t+t+j)-|\rem(T_0( a_{1})T_1\cdots( a_{t-2}) T_{t-2}(a_{t-1})W_{t-1}({a}_{t}) ( a_t+1,\dots,a_t+j-1)|$$
$$\Leftrightarrow$$
$$a_t+j<2(l-t+t+j)-(l-t+j-1)$$
$$\Leftrightarrow$$
\begin{align}a_t+j<2(l+j)-(l-t+j-1)\label{snail}\end{align}

From \eqref{formulainduction}, if $a_t\in2\Z$, for all $1\le j\le l_{t+1}-2$, and $j$ odd,
\begin{align}&\rem(T_0( a_{1})T_1\cdots( a_{t-2}) T_{t-2}(a_{t-1})W_{t-1}.(\s(a_{t}))({a}_{t})(a_t+1,a_t+2,\dots,a_t+j,a_t+j+1)\nonumber\\
&=W_1\cdots W_{t-1}.(s(a_t)).(a_t). (a_t+1,a_t+2,\dots, a_t+j,a_t+j+1)\nonumber
\end{align}
which means, by $\sum_{i=1}^t l_i=l-t$,

\begin{align*}&a_t+j<\nonumber\\
&<2(l-t+t+1+j)-|\rem(T_0( a_{1})T_1\cdots( a_{t-2}) T_{t-2}(a_{t-1})W_{t-1}.(s(a_t)).({a}_{t}). ( a_t+1,a_t+2,\dots, a_t+j-1)|\nonumber\\
&=2(l-t+t+j)+2-(l-t+j-1)-2\nonumber\\
&=2(l+j)-(l-t+j-1) \mbox{ which confirms \eqref{snail}}.
\end{align*}
\end{proof}

\begin{ex} We illustrate here the procedure in the proof of the previous theorem. Let $n=15$.
\begin{enumerate}
\item $l=18$ and ${\bf a}=(2,6,7,9, 14,16,22,23)\in SpT_{30}(\varpi_8)$ $t=8\le\min\{18,2n-l=30-18=12\}$, $l-t=10$.
It holds
$$T={\bf 2}(3,4){\bf 679}(11,12),{\bf 14,16},(17,18,19,20),{\bf 22,23},(25,26)\quad \mbox{ and }\redu(T)=\bf a.$$

One has,
\begin{align*}&W=[2n]\setminus\left({{\bf {a}}}\cup s({\bf {a}})\right)=\\
&[2n]\setminus\{2,6,7,9,14,16,22,23\}\cup\{ 1,5,8,10,13,15,21,24 \}\\
&=\{3,4,11,12,17,18,19,20,25,26,27,28,29,30\}.
\end{align*}

We consider the set $W^l_t$ of the first $l-t=18-8=10$ elements of $W$,
$$W^l_{t}=\{3,4,11,12,17,18,19,20,25,26\}.$$

Let $W^{l}_t\cap [a_{i-1},a_i]:=W_{i-1}=(a_{i-1})T_{i-1}\setminus\{a_{i-1}\}$, for $i=1,\dots,t+1$, with $a_0=0,~a_{t+1}=+\infty$.
It follows,
\begin{align*}&W^l_t\bigcap\big(\bigsqcup_{i=1}^{9} [a_{i-1},a_i])=
\{3,4,11,12,17,18,19,20,25,26\}\bigcap \big(\bigsqcup_{i=1}^7 [a_{i-1},a_i]\sqcup[a_7,a_8]\sqcup [a_8,+\infty]\big)\\
&=[3,4] \bigsqcup [11,12]\bigsqcup [17,18,19,20]\bigsqcup[25,26],\quad a_8=23.
\end{align*}

Hence,
$$T=2(3,4)679(11,12),14,16,(17,18,19,20),22, 23,(25,26) \quad \mbox{ and }\redu(T)=\bf a,$$

where $W^{l}_t\cap [a_{t},a_{t+1}]:=W_{t}=(a_t)T_{t}\setminus\{a_{t}\}=(25,26)$,

\item  For $l'=l-1=17$ and ${\bf a'}=(2,6,7,9, 14,16,22)={\bf a}\setminus \{a_t=23\}\in SpT_{30}(\varpi_7)$, $t'=t-1=7\le\min\{17,2n-17=30-17=13\}$, $l-t=10=(l-1)-(t-1)=l'-t'$.
Then put ${\bf a'}=(a'_1,\dots,a'_{t-1}):={\bf a}\setminus \{a_t\}$

$$s({\bf a'})=s({\bf a})\setminus s(a_t)=s({\bf a})\setminus \{\s(23)\}$$ and
\begin{align*}&W'=[2n]\setminus\left({{\bf {a'}}}\cup s({\bf {a'}})\right)=\\
&[2n]\setminus\left(\{2,6,7,9,14,16,22\}\sqcup\{ 1,5,8,10,13,15,21 \}\right)\\
&=\{3,4,11,12,17,18,19,20,23,24,25,26,27,28,29,30\}=W\sqcup \{a_t=23,\s(a_t)=s(23)=24\}
\end{align*}

We consider the set ${W'}^{l-1}_{t-1}$ of the first $l-t=(l-1)-(t-1)=18-8=10$ elements of $W'$,
$${W'}^{l-1}_{t-1}=\{3,4,11,12,17,18,19,20,23,24\}=(W^l_t\setminus\{25,26)\})\sqcup \{a_t,\s(a_t)\}.$$

Let ${W'}^{l-1}_{t-1}\cap [a'_{i-1},a'_i]:=W'_{i-1}=W_{i-1}=(a_{i-1})T_{i-1}\setminus\{a_{i-1}\}$, for $i=1,\dots,t'$, with $a_0=0,~a'_{t'+1}=+\infty$.

One has 
\begin{align*}&W^{l-1}_{t-1}\bigcap\bigsqcup_{i=1}^{8} [a'_{i-1},a'_i]=
\{3,4,11,12,17,18,19,20,{ a_t=23,\s(a_t)=24}\}\bigcap \big(\bigsqcup_{i=1}^7 [a_{i-1},a_i]\sqcup [a_7,+\infty)\big)\\
&=[3,4] \bigsqcup [11,12]\bigsqcup [17,18,19,20]\sqcup[23,24]\quad \mbox{ where } a_7=22
\end{align*}
and 
$T'=2(3,4)679(11,12),14,16,(17,18,19,20),22,(23,24),$ $\quad \redu(T')=\bf a'.$

\bigskip


\bigskip

\item For $l=18$ and $t'=6\le min\{18,2n-l=30-18=12\}$, $l-6=12$, $t'<t$, ${\bf a}'=(2,6,7,9, 14,16)\in SpT_{30}(\varpi_6)$.
One has, \begin{align*}&W=[2n]\setminus{{\bf {a}}}\cup s({\bf {a}})=\\
&[2n]\setminus(\{2,6,7,9,14,16\}\bigsqcup\{ 1,5,8,10,13,15 \})\\
&=\{3,4,11,12,17,18,19,20,21, 22,23,24,25,26,27,28,29,30\},
\end{align*}
$$W_{l-t'}=W_{12}=\{3,4,11,12,17,18,19,20,21,22,23,24, 25,26\}=W_{10}\cup\{25,26\}, ~a_7=+\infty$$
and
\begin{align*}&W_{l-t'}\bigcap\bigsqcup_{i=1}^7 [a_{i-1},a_i]=
\{3,4,11,12,17,18,19,20,21,22,23,24,25,26\}\bigcap \bigsqcup_{i=1}^7 [a_{i-1},a_i]\\
&=[3,4] \bigsqcup [11,12]\bigsqcup [17,18,19,20,{\bf 21,22,23,24,25,26}].
\end{align*}
\end{enumerate}
\end{ex}

 Previous theorem  characterizes the expanding of symplectic columns  where the intervals of length $>1$   as factors if they exist, consist solely of an even number followed with an odd number, the procedure in the proof provides the basic tool for the expanding of a generic symplectic column. Next lemma is supported by Corollary \ref{cor:cups}, equivalently, Theorem \ref{thm:symplecticdecomposition} with $k=1$ and  $p_0=p_1=0$.

\medskip
\begin{lem}\label{lem:reductionsurj} \cite{azreco} Let $l\in [0,2n]$ and $t\in[n]\cap 2\Z$ such that  $2\le t\le min\{l,2n-l\}$ and $l-t\in 2\Z$. Let $a_1\notin 2\Z$,  and   $\mathbf{A} = (a_1, a_1+1,\dots , a_1+t-2,a_1+t-1)\in SpT_{2n}(\varpi_t)$. Then
  $a_1\ge t+1$,  and
\begin{align*}
\rm{red}_{t}^{-1}:=\redu^{-1}_{|SpT_{2n}(\varpi_t)}:SpT_{2n}(\varpi_t)&\rightarrow SST_{2n}(\varpi_l)\quad\\
\mathbf{A}&\mapsto\rm{red}_{t}^{-1}(\mathbf{A})=T,\nonumber
\end{align*}
where \begin{align}T=(1,2,\dots, l_1)\mathbf{A}(a_1+t-1+1,\dots, a_1+t-1+l_{t+1})\label{lem:T},
\end{align}
and $l_1$, $l_{t+1}$  are nonnegative even numbers defined by
\begin{align}\label{lem0}
l_1=min\{a_1-t-1,l-t\},\quad
l_{t+1}=l-t-l_1.
\end{align}
\end{lem}
\begin{proof} A proof was given in \cite{azreco}. Alternatively, this result is immediate from the procedure given in the proof of Theorem \ref{thm:noconsecutive}. In fact since ${\bf A}=s({\bf A})$ and $\rem(T)=T\setminus {\bf A}$,
$$W=[2n]\setminus \left({\bf A}\sqcup s({\bf A})\right)=[2n]\setminus {\bf A}=\{1,\dots,a_1-1\}\sqcup \{a_1+t-1+1,\dots,2n\},$$
the expanding on the left of $\bf A$ is the largest set $\{1,\dots, x\}\subseteq \{1,\dots,a_1-1\}$ with $x\in 2\Z$ such that by Proposition \ref{prop:rem}, the entries of $\bf A$ are not removed. That is

\begin{align}a_1+t-2\ge 2(t-1+x)-x\Leftrightarrow a_1+t-2\ge 2t+x-2\Leftrightarrow  a_1-t-1\ge x, \mbox{ $x$ even }.\label{star}
\end{align}
Such $x\ge 0$ does exist because, by  Corollary \ref{cor:cups}, $a_1\ge t+1$ is a necessary and sufficient condition for the symplectic condition on $\bf A$.
Moreover,
$$a_1+t-4\ge 2(t-1+x)-x-4\Leftrightarrow a_1+t-4\ge 2(t-1-2+x)-x\Leftrightarrow a_1+t-4\ge 2(t-3+x)-x,$$
and similarly by Corollary \ref{cor:cups} for the remaining odd entries of $\bf A$.

To avoid removal pairs  in $\bf A$, the largest $x\in 2\Z$ that we can take is then $x=a_1-t-1$ \eqref{star} which means that increasing $x$ even, the  inequality \eqref{star} changes to

$$a_1+t-2<2(t-1+x+2)-(x+2)\Leftrightarrow a_1+t<2(t+x+1)-x$$
$$\Rightarrow a_1+t+(j-2)<2(t+x+j-1)-(x-j+2), \mbox{ for $2\le j\in 2\Z$.} $$

Henceforth, by Proposition \ref{prop:rem}, $(a_1+t-1+1,\dots, a_1+t-1+l_{t+1})\subseteq \rem(T)$, and  $\rem(T)=T\setminus\bf A$.

Hence, \eqref{lem0} and \eqref{lem:T}.
\end{proof}

\medskip
Thanks to Theorem \ref{thm:symplecticdecomposition}, we now combine this lemma with the previous theorem. More precisely, next lemma considers Theorem \ref{thm:symplecticdecomposition} with $k=1$ and $p_0>0$.

\begin{thm}\label{thm:xA} Let $l\in [0,2n]$ and $t_1+p_0\in[n]$, $t\in 2\Z$ such that  $2\le t_1+p_0\le min\{l,2n-l\}$ and $l-(t_1+p_0)\in 2\Z$. Let $a_1\notin 2\Z$, ${\bf A_1}=(a_1, a_1+1,\dots , a_1+t_1-2,a_1+t_1-1)$ and   $\mathbf{a} = x_1x_2\cdots x_{p_0} {\bf A_1}\in SpT_{2n}(\varpi_{t_1+p_0})$ such that $\s(x_i)\notin \mathbf{a} $, for $i=1,\dots,p_0$. Then
  $a_1\ge t_1+2p_0+1$,  and
\begin{align}
\rm{red}_{t}^{-1}:=\redu^{-1}_{|SpT_{2n}(\varpi_{t_1+p_0})}:SpT_{2n}(\varpi_{t_1+p_0})&\rightarrow SST_{2n}(\varpi_l)\quad \nonumber\\
\mathbf{a}&\mapsto\rm{red}_{t1+p_0}^{-1}(\mathbf{a})=T,\nonumber
\end{align}
where \begin{align}T=T_0(x_1)T_1\cdots (x_{p_0})  T_{p_0}\mathbf{\bf A_1}(a_1+t_1-1+1,\dots, a_1+t_1-1+ l_{t_1+p_0+1})\label{lem:T+}.
\end{align}
such that
\begin{enumerate}
\item  $\tilde l_1$  is the nonnegative even number defined by
\begin{align}\label{lem00+}
\tilde l_1=min\{a_1-t_1-2p_0-1,l-t_1-p_0\},\quad
\end{align}

\begin{align}\label{inverseredu:nofactors+}
&T_0=(1,2,\dots l_1),\\
&({x}_j)T_j=
\begin{cases}
 ({x}_j)( x_j+1,\dots, x_j+l_{j+1}),
  &
 \mbox{if } x_j\in 2\mathbb{Z}\\
 ({x}_j)( x_j+1+1,x_j+1+2,\dots, x_j+1+l_{j+1}),
 &\mbox{if } x_j\notin 2\mathbb{Z}\\
\end{cases},& 1\le j\le p_0, \label{inverseredu:nofactors2+}
\end{align}
with
$l_1,\dots, l_{p_0}$, $l_{p_0+1}$  nonnegative even numbers defined by
\begin{align}\label{l:nofactors+}
&l_1=
\begin{cases} min\{x_1-2, \tilde l_1\},& \mbox{if } x_1\in 2\mathbb{Z}\\
min\{ x_1-1, \tilde l_1\},&\mbox{if } x_1\notin 2\mathbb{Z},\\
\end{cases}
\end{align}
\begin{align}
&l_{j+1}=
\begin{cases}min\{x_{j+1}-x_j-2, \tilde l_1-\sum_{m=1}^j l_m\}, &\mbox{if }
x_j, x_{j+1}\in 2\mathbb{Z} \mbox{ or } x_j, x_{j+1}\notin 2\mathbb{Z}
\\
min\{{ x_{j+1}-x_j-1},~~ \tilde l_1-\sum_{m=1}^j l_m\},&\mbox{if }
x_j\in 2\mathbb{Z}, x_{j+1}\notin 2\mathbb{Z} \\
min\{x_{j+1}-x_j-3,  \tilde l_1-\sum_{m=1}^j l_m\},&\mbox{if }
 x_j\notin 2\mathbb{Z}, ~~~ x_{j+1}\in 2\mathbb{Z},
\end{cases},\quad \mbox{ for } 1\le j<p_0,\label{ells+}
\end{align}
 and \begin{align} l_{p_0+1}=\tilde l_1-\sum_{m=1}^{p_0}l_m,\end{align}

\item $ \displaystyle l_{t_1+p_0+1}=l-t_1-p_0- \sum_{m=1}^{p_0+1}l_m$.
\end{enumerate}

\end{thm}

\begin{proof}It is a consequence of Lemma \ref{lem:reductionsurj} and Theorem \ref{thm:noconsecutive}. Since $\bf a$  is symplectic, and $a_1$ is odd and $t_1$ even, $a_1-t_1-1\ge 2(t_1+p_0)-1\Leftrightarrow a_1\ge t_1+2p_0+1$.
 For $p_0=0$, we recover Lemma \ref{lem:reductionsurj}, $\tilde l_1=l_1$ and $\tilde l_{t+1}=l-t_1-l_1$. For $t_1=0$, $\bf A_1$ is empty and we recover Theorem \ref{thm:noconsecutive} with $\tilde l_1=l-p_0$.

From the procedure given in the proof of Theorem \ref{thm:noconsecutive}

$$s({\bf a})=\s(x_1\cdots x_{p_0})\sqcup {\bf A_1} \mbox { with $\s({\bf A_1})={\bf A_1}$ and } {\bf a}\cup \s({\bf a})=\s(x_1\cdots x_{p_0})\sqcup (x_1\cdots x_{p_0})\sqcup {\bf A_1}$$
and we want $\rem(T)=T\setminus {\bf a}$. Consider
$$W=[2n]\setminus \bigg({\bf a}\sqcup \s({\bf a})\bigg)=\bigg(\{1,\dots, a_1-1\}\setminus\big(\s(x_1\cdots x_{p_0})\sqcup (x_1\cdots x_{p_0})\big)\bigg)\sqcup \{a_1+t_1-1+1,\dots,2n\}.$$
So the expanding on the left hand side of $\bf a$ is the  set $X\subseteq \{1,\dots, a_1-1\}\setminus(s(x_1\cdots x_{p_0})\sqcup (x_1\cdots x_{p_0}))$ consisting of the first $x$ elements with  largest $|X|=x\in 2\Z$ so that by Proposition \ref{prop:rem}, the removal inequality is satisfied,
$\rem(X{\bf A_1})=X$,
$$a_1+t_1-2\ge 2(t_1-1+p_0+x)-x\Leftrightarrow a_1+t_1-2\ge 2t_1+2p_0+x-2\Leftrightarrow  a_1-t_1-2p_0\ge x$$
$$\Leftrightarrow a_1-t_1-2p_0-1\ge x\ge 0, \mbox{ $x$ even}.$$

To avoid removal pairs in $\bf A_1$, the largest $x$ that we can take is then $a_1-t_1-2p_0-1$. This justifies $\tilde  l_1$ \eqref{lem00+}, which  gives \eqref{lem:T+}, \eqref{inverseredu:nofactors+}, \eqref{inverseredu:nofactors2+} and \eqref{ells+}.
\end{proof}

\begin{ex}\label{ex:discussion}
\begin{enumerate}

        \item Let $n=7$, $l=11$, $t_1=2$, $p_0=1$, $t_1+p_0=3\le min\{l=11,14-11\}$, $l-t-p_0=11-3=8$ even.
        \begin{itemize}
        \item With  
        $$\mathbf{a} = 4 (9,10)\in SpT_{14}(\varpi_3),~~~x_1=4\in 2\Z$$ 
        
        Then $a_1-t_1-2p_0-1= 9-2-2-1=4$ and $\tilde l_1=min\{9-5=4,8\}=4$,
       $$ l_1=
 min\{x_1-2, \tilde l_1\}=min\{4-2,4\}=2\Rightarrow T_0=(1,2)$$
 and $$l_{p_0+1}=\tilde l_1-\sum_{m=1}^{p_0}l_m=4-2=2\Rightarrow T_1=(x_1+1,x_1+l_{p_0+1})=(5,6)$$
        
       Finally,
       $ \displaystyle l_{t_1+p_0+1}=l-t_1-p_0- \sum_{m=1}^{p_0+1}l_m=11-3-l_1-l_{p_0+1}=11-3-4=4$.

        Then $$T=(1,2){\bf 4}(5,6){\bf(9,10)}(11,12,13,14)\in
        SST_{14}(\varpi_{11}), \quad \redu(T)=\bf a.$$

\item With $\mathbf{a} = 8 (9,10)\in SpT_{14}(\varpi_3),$, one has
$$T=(1,2, 3,4)8(9,10)(11,12,13,14)\in
        SST_{14}(11), \quad \redu(T)=\bf a.$$
\end{itemize}

        \item Let $n=7$, $l=10$, $t_1=2$, $p_0=2$, $t_1+p_0=4\le min\{l=10,14-10\} $, $l-t_1-p_0=10-4=6$ even,and
        
        $${\bf a}=1, 7, (11, 12)\in SpT_{14}(\varpi_4),~~ x_1=1,~x_2=7\notin 2\Z,~~a_1- t_1-2p_0-1=11-2-4-1=4$$ 
        
       Then  $\tilde l_1=min\{4,6\}=4$ and $l_1= min\{x_1-1, \tilde l_1\}=min\{0,\tilde l_1\}=0\Rightarrow T_0=(),$
       $$l_2=min\{x_{2}-x_1-2, \tilde l_1-l_1\}=\min\{4,4-0\}=4\Rightarrow $$
       $$\Rightarrow T_1=(1+1+1, 1+3,1+4, 1+1+l_2)=(3,4,5,6),$$
        and
        $$l_{p_0+1}=\tilde l_1-\sum_{m=1}^{p_0}l_m=4-4=0\Rightarrow T_2=().$$
        
Finally,  $ \displaystyle l_{t_1+p_0+1}=l-t_1-p_0- \sum_{m=1}^{p_0+1}l_m=6-l_1-l_{p_0+1}=6-4=2$.
            Therefore 
            $$T={(\bf 1)}(3,4,5,6) {\bf(7) (11, 12)}(13,14)\in SST_{14}(\varpi_{10}).$$

         This example point $(2)$ may also be computed backwards as in \cite[Example 4.4.2]{watanabe} using combinatorial $R$-matrices. See the discussion in Section \ref{sec:final}.
\item Let $n=9$, $l=10$, $t_1=4$, $p_0=2$ and $t=t_1+p_0=4+2=6<min\{l=10,18-10\} $, $l-t_1-p_0=10-6=4$ even. Let
$${\bf a}=1,7,(11,12,13,14),~~~~\mbox{ where } ~~ a_1- t_1-2p_0-1=11-~4-4-1=2$$

Then 

$$T={\bf 1}(3,4){\bf 7}({\bf 11,12,13,14}) (15,16)\in SST_{18}(\varpi_{10}).$$

            \end{enumerate}

\end{ex}

\medskip

We now generalize a bit more by adjoining  a word in the conditions of Theorem \ref{thm:noconsecutive} to the right of ${\bf A_1}$. It  considers Theorem \ref{thm:symplecticdecomposition} with $k=1$ and $p_0,p_1>0$.

\begin{thm}\label{thm:xAy} Let $l\in [0,2n]$ and $t_1+p_0+p_1\in[n]$, $t\in 2\Z$ such that  $2\le t_1+p_0+p_1\le min\{l,2n-l\}$ and $l-(t_1+p_0+p_1)\in 2\Z$. Let
$$\mathbf{a} = x_1x_2\cdots x_{p_0} {\bf A_1}y_1\cdots y_{p_1}\in SpT_{2n}(\varpi_{t_1+p_0+p_1})$$
 where $a_1\notin 2\Z$, ${\bf A_1}=(a_1, a_1+1,\dots , a_1+t_1-2,a_1+t_1-1)$ and
 $$\s(x_1x_2\cdots x_{p_0}y_1\cdots y_{p_1})\cap\mathbf{a}=\emptyset .$$ Then
  $a_1\ge t_1+2p_0+1$,  and
\begin{align}
\rm{red}_{t_1+p_0+p_1}^{-1}:=\redu^{-1}_{|SpT_{2n}(\varpi_{t_1+p_0+p_1})}:SpT_{2n}(\varpi_{t_1+p_0+p_1})&\rightarrow SST_{2n}(\varpi_l)\nonumber\\
\mathbf{a}&\mapsto\rm{red}_{t_1+p_0+p_1}^{-1}(\mathbf{a})=T,\nonumber
\end{align}
where \begin{align}T=T_0(x_1)T_1\cdots  (x_{p_0})T_{p_0}\mathbf{A_1}Y_0(y_1)Y_1\cdots (y_{p_1})Y_{p_1}
\end{align}
is such that
\begin{enumerate}
\item  $\tilde l_1$ 
is the nonnegative even number defined by
\begin{align}
\tilde l_1=min\{a_1-t_1-2p_0-1,l-t_1-p_0-p_1\},\quad
\end{align}

\begin{align}\label{inverseredu:nofactors+++}
&T_0=(1,2,\dots l_1),\\
&({x}_j)T_j=
\begin{cases}
 ({x}_j)( x_j+1,\dots, x_j+l_{j+1}),
  &
 \mbox{if } x_j\in 2\mathbb{Z}\\
 ({x}_j)( x_j+1+1,x_j+1+2,\dots, x_j+1+l_{j+1}),
 &\mbox{if } x_j\notin 2\mathbb{Z}\\
\end{cases},& 1\le j\le p_0, \label{inverseredu:nofactors2++++}
\end{align}
with
$l_1,\dots, l_{p_0}$, $l_{p_0+1}$  nonnegative even numbers defined by
\begin{align}\label{l:nofactors++++}
&l_1=
\begin{cases} min\{x_1-2, \tilde l_1\},& \mbox{if } x_1\in 2\mathbb{Z}\\
min\{ x_1-1, \tilde l_1\},&\mbox{if } x_1\notin 2\mathbb{Z},\\
\end{cases}
\end{align}
\begin{align}
&l_{i+1}=
\begin{cases}min\{x_{j+1}-x_j-2, \tilde l_1-\sum_{m=1}^j l_m\}, &\mbox{if }
x_j, x_{j+1}\in 2\mathbb{Z} \mbox{ or } x_j, x_{j+1}\notin 2\mathbb{Z}
\\
min\{{ x_{j+1}-x_j-1},~~ \tilde l_1-\sum_{m=1}^j l_m\},&\mbox{if }
x_j\in 2\mathbb{Z}, x_{j+1}\notin 2\mathbb{Z} \\
min\{x_{j+1}-x_j-3,  \tilde l_1-\sum_{m=1}^j l_m\},&\mbox{if }
 x_j\notin 2\mathbb{Z}, ~~~ x_{j+1}\in 2\mathbb{Z},
\end{cases},\quad \mbox{ for } 1\le j<p_0,\label{ells+++}
\end{align}
and \begin{align}
 l_{p_0+1}=\tilde l_1-\sum_{m=1}^{p_0}l_m;\end{align}

\item put $y_0:=a_1+t_1-1\in 2\Z$, and define
\begin{align}\tilde l_{2}:=l-t_1-p_0-p_1- \sum_{m=1}^{p_0+1}l_m,
\end{align}

\begin{align}&Y_0=(y_0+1,y_0+2,\dots, y_0+ l_{ t_1+p_0+1}),\\
&({y}_j)Y_j=
\begin{cases}
 ({y}_j)( y_j+1,\dots, y_j+ l_{t_1+p_0+j+1}),
  &
 \mbox{if } y_j\in 2\mathbb{Z}\\
 ({y}_j)( y_j+1+1,y_j+1+2,\dots, y_j+1+ l_{t_1+p_0+j+1}),
 &\mbox{if } y_j\notin 2\mathbb{Z}\\
\end{cases},& 1\le j\le p_1, \label{inverseredu:nofactors2+++}
\end{align}
with
$ l_{t_1+p_0+j+1}$, $0\le j\le p_1$,
 nonnegative even numbers defined by
\begin{align}\label{l:nofactors+++}
& l_{t_1+p_0+1}=
\begin{cases} min\{y_1-y_0-2, \tilde l_2\},& \mbox{if } y_1\in 2\mathbb{Z}\\
min\{ y_1-y_0-1, \tilde l_2\},&\mbox{if } y_1\notin 2\mathbb{Z},\\
\end{cases}
\end{align}
\begin{align}
& l_{t_1+p_0+j+1}=
\begin{cases}min\{y_{j+1}-y_j-2, \tilde l_2-\sum_{m=1}^j  l_{t_1+p_0+m}\}, &\mbox{if }
y_j, y_{j+1}\in 2\mathbb{Z} \mbox{ or } y_, y_{j+1}\notin 2\mathbb{Z}
\\
min\{{ y_{j+1}-y_j-1},~~ \tilde l_2-\sum_{m=1}^j  l_{t_1+p_0+m}\},&\mbox{if }
y_j\in 2\mathbb{Z}, y_{j+1}\notin 2\mathbb{Z} \\
min\{y_{j+1}-y_j-3,  \tilde l_2-\sum_{m=1}^j  l_{t_1+p_0+m}\},&\mbox{if }
 y_j\notin 2\mathbb{Z}, ~~~ ,y_{j+1}\in 2\mathbb{Z},
\end{cases},\quad \mbox{ for } 1\le j<p_1,\label{ells++++}
\end{align}
and

\begin{align} l_{t_1+p_0+p_1+1}=\tilde l_2-\sum_{m=1}^{p_1}l_{t_1+p_0+m}.\end{align}
 \qquad\qquad\qquad\qquad\qquad\qquad\qquad\qquad\qquad\qquad\qquad\qquad\qquad\qquad\qquad\qquad\qquad\qquad\qquad
 \qquad$\Box$
\end{enumerate}
\end{thm}

\medskip
\begin{ex}

Let $n=9$, and  
$$\mathbf{a} = 8 (9,10)(11, 16)\in SpT_{18}(\varpi_5), \mbox{  $x_1=8\in 2\Z$ and $y_1=11\notin2\Z$, $y_2=16\in 2\Z$}.$$

Let $l\in [0,18]$, $l=11$ and $p_0=1$, $t_1=p_1=2$, $5=t_1+p_0+p_1\in[n]$, $t_1\in 2\Z$ such that  $2\le t_1+p_0+p_1\le min\{l,2n-l\}$ and $l-(t_1+p_0+p_1)=6\in 2\Z$. One has $a_1=9\notin 2\Z$,   such that $\s(x_1)=7, \s(y_j)\notin \mathbf{a} $, for $j=1,2$. Then
  $a_1\ge t_1+2p_0+1=5$, and the largest is $x=4$. Then
  
  $$4=\tilde l_1=min\{a_1-t_1-2p_0-1,l-t_1-p_0-p_1\}=min\{9-5=4,11-5=6\}$$
  $$4=l_1=min\{x_1-2, \tilde l_1\}=min\{6,4\}=\tilde l_1\Rightarrow l_{p_0+1}=l_2=0$$
and $$T_0=(1234),~ ~T_1=().$$

Put $y_0=a_1+t_1-1=9+2-1=10\in 2\Z$, and define
\begin{align}\tilde l_{2}:=l-t_1-p_0-p_1- \sum_{m=1}^{p_0+1}l_m=6-l_1-l_2=2,
\end{align}
Then  $$l_{t_1+p_0+1}=
min\{ y_1-y_0-1, \tilde l_2\}=\{11-10-1,2\}=0\Rightarrow Y_0=()$$
$$l_{t_1+p_0+p_1}=min\{y_{2}-y_1-3,  \tilde l_2-  l_{t_1+p_0+1}\}=\min\{16-11-3=2,2\}=2,
 ~\Rightarrow Y_1=(y_1+1+1,y_1+2+1)=(13,14),$$
 
 and $$l_{t_1+p_0+p_1+1}=\tilde l_2-\sum_{m=1}^{p_1}l_{t_1+p_0+m}=2-2=0\Rightarrow Y_2=()$$
Therefore,
  $$T=(1234){\bf 8(9,10)11}(13,14){\bf 16}\in SST_{2n}(\varpi_{11}).$$
\end{ex}

\bigskip

Next theorem was partially stated in a slightly different manner \cite{azreco} that  we now uniformize and prove. It considers symplectic columns in the conditions of Theorem   \ref{thm:symplecticdecomposition}
 with $k=2$ and $p_0=p_1=p_2=0$.

\begin{thm}\label{lem:reductionsurj2} Let $l\in [0,2n]$ and  $ t_1\neq 0, t_2\neq 0\in 2\Z$ such that $t=t_1+t_2\in[n]$,
$$4\le t=t_1+t_2\le min\{l,2n-l\}$$ and $l-t\in 2\Z$. Let 
\begin{align*}\mathbf{a} &={\bf A_1}\bigoplus{\bf A_2}\in SpT_{2n}(\varpi_{t_1+t_2})
\end{align*}
 where $a_1,a_2\notin 2\Z$, ${\bf A_1}=({a_1}, a_1+1,\dots , a_1+t_1-2,a_1+t_1-1)$ and ${\bf A_2}=({a_2}, a_2+1,\dots , a_2+t_2-2,a_2+t_2-1)$.
Then, $a_1\ge t_1+1$, $a_2\ge 2t_1+t_2+1$ and
\begin{align}
\rm{red}_{t}^{-1}:=\redu^{-1}_{|SpT_{2n}(\varpi_t)}:SpT_{2n}(\varpi_t)&\rightarrow SST_{2n}(\varpi_l)\nonumber\\
\mathbf{a}&\mapsto\rm{red}_{t}^{-1}(\mathbf{a}):=T,\nonumber
\end{align}
where \begin{align}\label{xx}T&=(1,\dots, l_1){\bf {A_1}}( a_1+t_1-1+1,\dots,  a_1+ t_1-1+l_2){\bf{A_2}}( {a_2}+ t_2-1+1,\dots,  a_2+ t_2-1+l_{t+1}),
\end{align}
with $l_1, l_2$, $l_{t+1}$  nonnegative even numbers defined by
\begin{align}\label{prop02}
&l_1=min\{a_1-t_1-1,a_2-2t_1-t_2-1,l-t\},\\
&l_2=min\{(a_2-2t_1-t_2-1)-l_1,l-t-l_1\}\\
& l_{t+1}=l-t-\sum_{i=1}^{2}l_i.
\end{align}
\end{thm}

\begin{proof}
Let ${\bf A}={\bf A_1}\sqcup {\bf A_2}$.
One has $s({\bf A})=\bf A$.
 Let
 $$W=[2n]\setminus \left({\bf A}\sqcup s({\bf A})\right)=\{1,\dots,a_1-1\}\sqcup \{a_1+t_1,\dots,a_2-1\}\sqcup \{a_2+t_2,\dots, 2n\}$$

 The expanding on the left of $\bf A_2$ is the  set $X\subseteq \{1,\dots,a_2-1\}\setminus {\bf A_1}$ with largest cardinality   $x=|X|\in 2\Z$ such that by Proposition \ref{prop:rem}, the entries of $\bf A_2$ are not removed by the removal operation, that is, we want

\begin{align}a_2+t_2-2\ge 2(t-1+x)-x\Leftrightarrow a_2+t_2-2\ge 2t_1+2t_2+x-2\Leftrightarrow  a_2-2t_1-t_2-1\ge x, \mbox{ $x$ even }.\label{star3}
\end{align}
Such $x\ge 0$ does exist because $a_2\ge 2t_1+t_2+1$ is imposed by the symplectic condition on $\bf A_2$.
To avoid removal pairs  in $\bf A_2$, the largest $x\in 2\Z$ that we can take is then $x=a_2-2t_1-t_2-1$ which means that increasing $x$ even, the  inequality \eqref{star} changes to

$$a_2+t_2-2<2(t_2-1+x+2)-(x+2)\Leftrightarrow a_2+t_2<2(t_2+x+1)-x$$
$$\Rightarrow a_2+t+(j-2)<2(t_2+x+j-1)-(x-j+2), \mbox{ for $2\le j\in 2\Z$.} $$

Henceforth, by Proposition \ref{prop:rem}, $(a_2+t_2-1+1,\dots, a_2+t_2-1+l_{t+1})\subseteq \rem(T)$.

Now the set $X$ with cardinality $x$ has to be distributed between $\{1,\dots, a_1-1\}$ and $\{a_2+t_1,\dots, a_2-1\}$ and we define two sets $X_1$ and $X_2$. The part that goes to
$\{1,\dots, a_1-1\}$  is by Lemma \ref{lem:reductionsurj} bounded by $a_1-t_1-1$. Let $X_1=\{1,\dots,l_1\}$   and $l_1=min\{a_1-t_1-1,a_2-2t_1-t_2-1,l-t\}$.
If $l_2=0$ it means $l_1=a_2-2t_1-t_2-1$ or $l_1=l-t$ and nothing is left and $X_2=\emptyset$. Otherwise, we define
$X_2=\{a_1+t_1-1+1,\dots, a_1+t_1-1+l_2\}$ with $l_2=min\{a_2-2t_1-t_2-1)-l_1,l-t-l_1\}$.

Therefore
 \eqref{xx} and
$\rem(T)=T\setminus {\bf A}$.
\end{proof}

\bigskip
\begin{ex}Let $n=12$, $k=2$.
\begin{enumerate}
\item  Let $l=12$, $t_1=2<t_2=4$, $t=6\le min\{12, 24-12\},$ $l-t=6$, and

$${\bf a}=(9,10)(17, 18, 19, 20)\in SpT_{24}(\varpi_6),$$

$$a_1-t_1-1=9-2-1=6$$ 
$$a_2-2t_1-t_2-1=8.$$ 

Then 
$$6=l_1=min\{a_1-t_1-1,a_2-2t_1-t_2-1,l-t\}=\min\{6, 8,6\}$$  
$$0=l_2=min\{(a_2-2t_1-t_2-1)-l_1,l-t-l_1\}=min\{8-6,6-6=0\}$$

$$l_{t+1}=l-t-l_1-l_2=0$$ and

$$T=(123456){\bf(9,10)(17,18,19,20)}\in SST_{24}(\varpi_{12})$$

\item $l=16$, $t=6\le min\{16, 24-16\},$ $l-t=10$, and $${\bf a}=(9,10)(17, 18, 19, 20)\in SpT_{24}(\varpi_6).$$  Then
    
    $$T=(123456){(\bf 9,\bf 10)}(11,12){(\bf 17,18,19,20)}(21,22)\in SST_{24}(\varpi_{16})$$



     \item  Let $l=16$, $t=t_1+t_2=6\le min\{16, 24-16\},$ $l-t=10$,
     $${\bf a}=(9,10,11,12)(19,20)\in SpT_{24}(\varpi_6),$$
     
     $a_2-2t_1-t_2-1=19-8-3=8$, $a_1-t_1-1=9-5=4$. 
     
     Then $l_1=4$ and $l_2=4$, $l_{t+1}=l-t-8=0$,

     $$T=(1234){\bf (9,10,11,12)}(13,14,15,16){\bf(19,20)}\in SST_{24}(\varpi_{16})$$

\item   Let $l=12$, 
$${\bf a}=(9,10,11,12)(17,18, 19, 20)\in SpT_{24}(\varpi_8),$$ where $a_2-2\times 4- 4-1=17-13=4$ and $a_1-t_1-1=4$.
Then $l_1=a_1-t_1-1=4=a_2-2t_2-t_1-1$, $l_2=0$, $l-t= 12-8=4$, $l_{t+1}=0$.

Therefore,
$$T=(1234)\bf a.$$

\item $l=16$, $t=8\le min\{16, 24-16\},$ $l-t=8$,

${\bf a}=(9,10,11,12)(15,16, 17, 18)\in SpT_{24}(\varpi_8)$, 
$$a_1\ge t_1+1\Leftrightarrow 9-5-1=3,~~a_2\ge 2t_1+t_2+1\Leftrightarrow 15- (8+4+1)=2>0$$
~

Then, $l_1=2=\min\{a_1-t_1-1,a_2-2t_1-t_2-1, l-t\}$, $l_2=0$ and $l_{t+1}=8-l_1=6$
$$T=(12){\bf a}(19,20,21,22,23,24)\in SST_{24}(\varpi_{16})$$

\item $l=14$, $t=10$, $l-t=4$, ${\bf a}=(9,10,11,12)(15,16, 17, 18, 19,20)\in SpT_{24}(\varpi_{10})$,
$l_1=l_2=0$ because $a_2-2t_1-t_2-1=15-8-6-1=0$. Then $l_{11}=4$,
and 
$$T={\bf a}(21,22,23,24)\in SST_{24}(\varpi_{14}).$$
\end{enumerate}
\end{ex}

\medskip

\medskip
 We now generalize the two previous theorems.

\begin{thm}\label{thm:intervals} Let $l\in [0,2n]$ and   $ t_i\neq 0\in 2\Z$, $i=1,\dots,k$, such that $t=t_1+\cdots+t_k\in[n]$,
$$2\le t=\sum_{i=1}^kt_i\le min\{l,2n-l\}$$ and $l-t\in 2\Z$.
Let
\begin{align*}\mathbf{a} &=\bigoplus_{i=1}^k{\bf A_i}\in SpT_{2n}(\varpi_{t})
\end{align*}
where for $i=1,\dots,k$, ${a_i}\notin 2\Z$ and
\begin{align*}&{\bf A_i}=({a_i}, a_i+1,\dots , a_i+t_i-2,a_i+t_i-1).
\end{align*} Then $a_i\ge 2(t_1+\cdots+t_{i-1})+t_i+1$, $1\le i\le k$, and

\begin{align*}
\rm{red}_{t}^{-1}:=\redu^{-1}_{|SpT_{2n}(\varpi_t)}:SpT_{2n}(\varpi_t)&\rightarrow SST_{2n}(\varpi_l)\quad\\
\mathbf{a}&\mapsto\rm{red}_{t}^{-1}(\mathbf{a}):=T,\nonumber
\end{align*}
where \begin{align*}T&=(1,\dots, l_1){\bf {A_1}},( a_1+ 1+t_1-1,\dots,  a_1+ 1+t_1-1+l_2){\bf{A_2}}
( a_2+ 1+t_2-1,\dots,  a_2+ 1+t_2-1+l_3)\\
&\cdots{\bf A_{k-1}}( a_{k-1}+ 1+t_{k-1}-1,\dots,  a_{k-1}+ 1+t_{k-1}-1+l_{k})\\
&{\bf {A_k}}( {a_k}+ 1+t_k-1,\dots,  a_k+ 1+t_k-1+l_{t+1})
\end{align*}
with
$l_1,\dots, l_k$, $l_{t+1}$  nonnegative even numbers recursively defined for $i=1,\dots,k$, by
\begin{align}
&l_i=min\{~\displaystyle (a_j-2\sum_{m=1}^{j-1}t_m-t_{j}-1)-\sum_{m=1}^{i-1}l_m,~ i\le j\le k,~~l-t-\sum_{m=1}^{i-1}l_m\},\\
& \mbox{ and}\nonumber\\
&l_{t+1}=l-t-\sum_{i=1}^{k}l_i.
\end{align}

\end{thm}

\bigskip
\begin{ex}Let $n=20$ and $k=3$.
\begin{enumerate}
\item Let $l=16$, $t_1=t_2=t_3=4$, $t=t_1+t_2+t_3=12\le min\{l, 2n-l\}$,  $l-t=4$, and

$${\bf a}=(9,10,11,12)(15,16,17,18)(21,22,23,24)=\bigoplus_{i=}^3 {\bf A_i}\in SpT_{40}(\varpi_{12}).
$$

Then
$$a_1-t_1-1=9-4-1=4$$
$$ a_2-2t_1-t_2=15-8-4=3$$
$$a_3-2(t_1+t_2)-t_3-1=21-16-4-1=0\Rightarrow l_3= l_2=l_1=0$$ 

$$\displaystyle 0= l_1=min\{a_j-2\sum_{m=1}^{j-1}t_m-t_j-1, ~1\le j\le 3,~~l-t=4\}=min\{4,3,0,l-t\}$$

$$0=l_2=min\{~\displaystyle (a_j-2\sum_{m=1}^{j-1}t_m-t_{j}-1)-l_1,~ 2\le j\le 3,~~l-t-l_1\}=min\{3,0,l-t\}$$
$$0=l_3=min\{~\displaystyle (a_3-2(t_1+t_2)-t_{3}-1)-l_1-l_2,~~l-t-l_1-l_2\}=min\{0,l-t\}$$
 $$l_{t+1}=l-t-l_1-l_2-l_3=l-t=4$$

and
$$T={\bf a}(25,26,27,28)\in SST_{40}(\varpi_{16}).$$
\item Let $l=18$, $t_1=4$, $t_2=2$, $t_3=6$, $t=t_1+t_2+t_3=12\le min\{l, 2n-l\}$, $l-t=6$,and 

$${\bf a}=(9,10,11,12)(15,16)(21,22,23,24,25,26)=\bigoplus_{i=}^3 {\bf A_i}\in SpT_{40}(\varpi_{12}).$$

Then
$$a_1-t_1-1=9-5=4$$
$$a_2-2t_1-t_2-1=15-8-2-1=4$$
$$a_3-2(t_1+t_2)-t_3-1=21-12-6-1=2$$
and

$$2=l_1=a_3-2(t_1+t_2)-t_3-1=\min\{9-5=4,15-8-2-1=4,a_3-2(t_1+t_2)-t_3-1=21-12-6-1=2,6\}$$
$$0=l_2=min\{15-8-2-1-l_1=4-l_1=a_2-2t_1-t_2-1-l_1=2, a_3-2(t_1+t_2)-t_3-1-l_1=2-2=0, 6-2\} $$

$$\Rightarrow l_3=0=\{(21-12-6-1=2)-l_1-l_2=0, 6-2\}$$

$l_4=l-t-l_1-l_2-l_3=6-2=4$

Then
$$T=(1,2){\bf(9,10,11,12)(15,16)}{\bf (21,22,23,24,25,26)}(27,28, 29,30)\in SST_{40}(\varpi_{18}).$$
$$\rem(T)=(1,2)(27,28,29,30)$$
\end{enumerate}
\end{ex}
\bigskip

Mixing Theorem \ref{thm:noconsecutive} and Theorem \ref{thm:intervals} we get the full formula for the inverse reduction.

\begin{thm}\label{thm:reductiexp} Let $l\in [0,2n]$ and    $ t_i\neq 0\in 2\Z$, $i=1,\dots,k$, $p_i\ge 0$, $i=0,\dots,k$, for some $k\ge 0$, and $\displaystyle t=\sum_{j=0}^kp_j + \sum_{j=1}^kt_j   \in [n]$ such that
$$ t\le min\{l,2n-l\}$$ and
$l-t\in 2\Z$.

Let
\begin{align}
{\bf a}=
X^{(0)}{\bf A_1} X^{(1)}\cdots X^{(k-1)} {\bf A_k} X^{(k)} \in SpT_{2n}(\varpi_t)
\end{align} with $X^{(i)}=(x_1^{(i)})\cdots (x_{p_i}^{(i)})$, $0\le i\le k$, such that
$$\s(X^{(0)} X^{(1)}\cdots X^{(k-1)}X^{(k)})\cap {\bf a}=\emptyset$$ and,  for $i=1,\dots,k$, ${\bf A_i}=({a_i}, a_i+1,\dots , a_i+t_i-2,a_i+t_i-1)$, with $a_i\notin 2\Z$,
are the maximal nonempty intervals such that $\s({\bf A_i})=\bf A_i$,
satisfying
\begin{align}
&a_i\ge 2t^{(i-1)}+t_i+1, ~\mbox{ with $\displaystyle t^{(i-1)}=\sum_{j=1}^{i-1}t_j+\sum_{j=0}^{i-1}p_j$, $t^{(0)}=p_0$, and $t^{(k)}:=t$.}
\end{align}

Then, 
\begin{align}
\rm{red}_{t}^{-1}:=\redu^{-1}_{|SpT_{2n}(\varpi_t)}:SpT_{2n}(\varpi_t)&\rightarrow SST_{2n}(\varpi_l)\nonumber\\
\mathbf{a}&\mapsto\rm{red}_{t}^{-1}(\mathbf{a}):=T,\nonumber
\end{align}
where
\begin{align}&T=T^{(0)}_0(x_1^{(0)})T^{(0)}_1\cdots (x_{p_0}^{(0)}) T^{(0)}_{p_0}\mathbf{A_1}T^{(1)}_0(x_1^{(1)})T^{(1)}_1\cdots (x_{p_1}^{(1)}) T^{(1)}_{p_1}\cdots\nonumber \\
&T^{(k-1)}_0(x_1^{(k-1)})T^{(k-1)}_1\cdots (x_{p_{k-1}}^{(k-1)}) T^{(k-1)}_{p_{k-1}}{\bf A_k}T^{(k)}_0(x_1^{(k)})T^{(k)}_1\cdots (x_{p_{k}}^{(k)}) T^{(k)}_{p_{k}}
\end{align}

is such that
\begin{enumerate}
\item { for $i=0$, consider $\tilde l_{1}$
the nonnegative even number defined by
\begin{align}  \label{i=0}
\tilde l_{1}=min\{a_j- 2t^{(j-1)}-t_j-1,~~{1}\le j\le k,~
l-t\},
\end{align}
}

\begin{align}
&T^{(0)}_0=(1,2,\dots l_1),\\
&({x}^{(0)}_j)T^{(0)}_j=
\begin{cases}
 ({x}^{(0)}_j)( x^{(0)}_j+1,\dots, x^{(0)}_j+l_{j+1}),
  &
 \mbox{if } x^{(0)}_j\in 2\mathbb{Z}\\
 ({x}^{(0)}_j)( x^{(0)}_j+1+1,x^{(0)}_j+1+2,\dots, x^{(0)}_j+1+l_{j+1}),
 &\mbox{if } x^{(0)}_j\notin 2\mathbb{Z}\\
\end{cases},& 1\le j\le p_0, 
\end{align}
\noindent with
 { $l_j$, $1\le j\le t^{(0)}+1$,} nonnegative even numbers defined by
\begin{align}
&l_1=
\begin{cases} min\{x^{(0)}_1-2, \tilde l_1\},& \mbox{if } x^{(0)}_1\in 2\mathbb{Z}\\
min\{ x^{(0)}_1-1, \tilde l_1\},&\mbox{if } x^{(0)}_1\notin 2\mathbb{Z},\\
\end{cases}
\end{align}
\begin{align}
&l_{j+1}=
\begin{cases}min\{x^{(0)}_{j+1}-x^{(0)}_j-2, \tilde l_1-\sum_{m=1}^j l_m\}, &\mbox{if }
x^{(0)}_j, x^{(0)}_{j+1}\in 2\mathbb{Z} \mbox{ or } x^{(0)}_j, x^{(0)}_{j+1}\notin 2\mathbb{Z}
\\
min\{{ x^{(0)}_{j+1}-x^{(0)}_j-1},~~ \tilde l_1-\sum_{m=1}^j l_m\},&\mbox{if }
x^{(0)}_j\in 2\mathbb{Z}, x^{(0)}_{j+1}\notin 2\mathbb{Z} \\
min\{x^{(0)}_{j+1}-x^{(0)}_j-3,  \tilde l_1-\sum_{m=1}^j l_m\},&\mbox{if }
 x^{(0)}_j\notin 2\mathbb{Z}, ~~~ x^{(0)}_{j+1}\in 2\mathbb{Z}
\end{cases},\quad \mbox{ for } 1\le j<p_0,
\end{align}
and
\begin{align}
&l_{t^{(0)}+1}=l_{p_0+1}=\tilde l_1-\sum_{m=1}^{t^{(0)}}l_m;
\end{align}

\item for $i=1,\dots,k-1$, put $x^{(i)}_0:=a_i+t_i-1\in 2\Z$,  and consider $\tilde l_{i+1}$
the nonnegative even number defined recursively by

\begin{align}\oo  \label{i+1}
&\tilde l_{i+1}=\nonumber\\
&=min\bigg(\{a_j- 2t^{(j-1)}-t_j-1-\sum_{m=1}^{t^{(0)}+1}l_{m}-\sum_{m=t^{(0)}+t_1+1}^{t^{(1)}+1}l_{m}-\dots -\sum_{m=t^{(i-2)}+t_{i-1}+1}^{t^{(i-1)}+1}l_{m},~~i+1\le j\le k\}\nonumber\\
&\bigcup
\{l-t-\sum_{m=1}^{t^{(0)}+1}l_{m}-\sum_{m=t^{(0)}+t_1+1}^{t^{(1)}+1}l_{m}-\dots -\sum_{m=t^{(i-2)}+t_{i-1}+1}^{t^{(i-1)}+1}l_{m}\}\bigg),
~~ t^{(-1)}=t_0~:=0.
\end{align}

\begin{align}& T^{(i)}_0=(x^{(i)}_0+1,x^{(i)}_0+2,\dots, x^{(i)}_0+ l_{ t^{(i-1)}+t_i+1}),\\
&({x}^{(i)}_j)T^{(i)}_j=
\begin{cases}
 (x^{(i)}_j)( x^{(i)}_j+1,\dots, x^{(i)}_j+ l_{t^{(i-1)}+t_i+j+1}),
  &
 \mbox{if } x^{(i)}_j\in 2\mathbb{Z}\\
 ({x}^{(i)}_j)( x^{(2)}_j+1+1,x^{(i)}_j+1+2,\dots, x^{(i)}_j+1+ l_{t^{(i-1)}+t_i+j+1}),
 &\mbox{if } x^{(i)}_j\notin 2\mathbb{Z}\\
\end{cases},& 1\le j\le p_i, 
\end{align}
with
$ l_{t^{(i-1)}+t_i+j+1}$, $0\le j\le p_i$,
 nonnegative even numbers defined by
\begin{align}
& l_{t^{(i-1)}+t_i+1}=
\begin{cases} min\{x^{(i)}_1-x^{(i)}_0-2, \tilde l_{i+1}\},& \mbox{if } x^{(i)}_1\in 2\mathbb{Z}\\
min\{ x^{(i)}_1-x^{(i)}_0-1, \tilde l_{i+1}\},&\mbox{if } x^{(i)}_1\notin 2\mathbb{Z},\\
\end{cases}
\end{align}
\begin{align}
&l_{t^{(i-1)}+t_i+j+1}=
\begin{cases}min\{x^{(i)}_{j+1}-x^{(i)}_j-2, \tilde l_{i+1}-\sum_{m=1}^j  l_{t^{(i-1)}+t_i+m}\}, &\mbox{if }
x^{(i)}_j, x^{(i)}_{j+1}\in 2\mathbb{Z} \mbox{ or } x^{(i)}_j, x^{(i)}_{j+1}\notin 2\mathbb{Z}
\\
min\{{ x^{(i)}_{j+1}-x^{(i)}_j-1},~~ \tilde l_{i+1}-\sum_{m=1}^j  l_{t^{(i-1)}+t_i+m}\},&\mbox{if }
x^{(i)}_j\in 2\mathbb{Z}, x^{(i)}_{j+1}\notin 2\mathbb{Z} \\
min\{x^{(i)}_{j+1}-x^{(i)}_j-3,  \tilde l_{i+1}-\sum_{m=1}^i  l_{t^{(i-1)}+t_i+m}\},&\mbox{if }
 x^{(i)}_j\notin 2\mathbb{Z}, ~~~ ,x^{(i)}_{j+1}\in 2\mathbb{Z}
\end{cases},\nonumber\\
&\quad \mbox{ for } 1\le j<p_i,
\end{align}
and
\begin{align}  l_{t^{(i)}+1}=\tilde l_{i+1}-\sum_{m=1}^{p_i} l_{t^{(i-1)}+t_i+m}
\end{align}
\bigskip

\bigskip

\item { for $i=k$}, put $x^{(k)}_0:=a_k+t_k-1\in 2\Z$,  and with $\tilde l_{k+1}$
the nonnegative even number defined  by

\begin{align}\tilde l_{k+1}:=l-t-\sum_{m=1}^{t^{(0)}+1}l_{m}-\sum_{m=t^{(0)}+t_1+1}^{t^{(1)}+1}l_{m}-\dots -\sum_{m=t^{(k-2)}+t_{k-1}+1}^{t^{(k-1)}+1}l_{m}\end{align}

\begin{align}&T^{(k)}_0=(x^{(k)}_0+1,x^{(k)}_0+2,\dots, x^{(k)}_0+ l_{t^{(k)}+1}),\\
& ({x}^{(k)}_j)T^{(k)}=
\begin{cases}
 ({x}^{(k)}_j)( x^{(k)}_j+1,\dots, x^{(k)}_j+ l_{t^{(k-1)}+j+1}),
  &
 \mbox{if } x^{(k)}_j\in 2\mathbb{Z}\\
 ({x}^{(k)}_j)( x^{(k)}_j+1+1,x^{(k)}_j+1+2,\dots,
x^{(k)}_j+1+ l_{t^{(k-1)}t_k+j+1}),
 &\mbox{if } x^{(k)}_j\notin 2\mathbb{Z}\\
\end{cases},& 1\le j\le p_k, 
\end{align}
with
$ l_{t^{(k-1)}+t_k+1}$, $0\le j\le p_k$,
 nonnegative even numbers defined by
\begin{align}
& l_{t_1+t_2+p_0+p_1+1}=
\begin{cases} min\{x^{(k)}_1-x^{(k)}_0-2, \tilde l_{k+1}\},& \mbox{if } x^{(k)}_1\in 2\mathbb{Z}\\
min\{ x^{(k)}_1-x^{(k)}_0-1, \tilde l_{k+1}\},&\mbox{if } x^{(k)}_1\notin 2\mathbb{Z},\\
\end{cases}
\end{align}
\begin{align}
&l_{t^{(k-1)}+t_k+j+1}=
\begin{cases}min\{x^{(k)}_{j+1}-x^{(k)}_j-2, \tilde l_{k+1}-\sum_{m=1}^j  l_{t^{(k-1)}+t_k+m}\}, &\mbox{if }
x^{(k)}_j, x^{(k)}_{j+1}\in 2\mathbb{Z} \mbox{ or } x^{(k)}_j, x^{(k)}_{j+1}\notin 2\mathbb{Z}
\\
min\{{ x^{(k)}_{j+1}-x^{(k)}_j-1},~~ \tilde l_{k+1}-\sum_{m=1}^j  l_{t^{(k-1)}+t_k+m}\},&\mbox{if }
x^{(k)}_j\in 2\mathbb{Z}, x^{(k)}_{j+1}\notin 2\mathbb{Z} \\
min\{x^{(k)}_{j+1}-x^{(k)}_j-3,  \tilde l_{k+1}-\sum_{m=1}^j  l_{t^{(k-1)}+t_k+m}\},&\mbox{if }
 x^{(k)}_j\notin 2\mathbb{Z}, ~~~ ,x^{(k)}_{j+1}\in 2\mathbb{Z}
\end{cases},\nonumber\\
&\quad \mbox{ for } 1\le j<p_k, \mbox{ and}
\end{align}
\noindent  \begin{align}  l_{t^{(k)}+1}=\tilde l_{k+1}-\sum_{m=1}^{p_k} l_{t^{(k-1)}+t_k+m}.\end{align} \qquad\qquad\qquad\qquad\qquad\qquad\qquad\qquad\qquad\qquad\qquad\qquad\qquad\qquad\qquad\qquad\qquad\qquad\qquad\qquad\qquad$\Box$
\end{enumerate}

\end{thm}

\section{Final remarks: the inverse reduction map as a composition of several maps and combinatorial $R$-matrices}\label{sec:final}
In \cite[Sections 4.3,~4.4]{watanabe} the reduction map is given as composition of several maps, as below in Theorem \ref{thm:composition}, and among them combinatorial  $R$-matrices \cite{rmatrix} and a reduction map on lower length  columns.
\begin{thm}\label{thm:composition}\cite[Corollary 4.4.3]{watanabe} For a fixed $n\in\mathbb{N}$, let $l\in[0,2n]$ and $l\ge 2$. Then
\begin{align}\redu=\pi\circ\bigwedge\circ(K^{-1},id)\circ R\circ(\redu,id)\circ R\circ (K,id)\circ \bigvee.
\end{align}
For the definition of these maps, we refer the reader to \cite[Sections 4.3,~4.4]{watanabe}.
\end{thm}

This is conceptually very interesting because, in particular, it allows to reduce the computation  to  columns  of lower length \cite[Proposition 4.4.1]{watanabe}. Combinatorial $R$-matrices are bijections and so we could use the inverse of these maps to compute the inverse reduction.
However one still  needs to compute the inverse reduction of shorter symplectic columns  and thereby it is useful to  have at hand the inverse.

For a fixed $n\in\mathbb{N}$, let $l\in [0,2n]$ and $t\in[0,n]$ such that  $0\le t\le min\{l,2n-l\}$ and $l-t\in 2\Z$,
\begin{align}
&\redu^{-1}:SpT_{2n}(\varpi_t)\rightarrow SST_{2n}(\varpi_l)\nonumber\\
&\mathbf{a}\mapsto\redu^{-1}(\mathbf{a}),\nonumber
\\
&\mbox{and } \redu^{-1}={\bigvee}^{-1}\circ (K^{-1},id)\circ R^{-1}\circ (\redu^{-1},id)\circ R^{-1}\circ (K,id)\circ {\bigwedge}^{-1}\circ \pi^{-1}.
\end{align}

Our method in Theorem \ref{thm:reductiexp} gives  a general procedure on how to expand a given general symplectic column \eqref{generalsymplecticdecompose}
to a column of an admissible given length as in the conditions above. A compromise is a combination of our procedures with the composition above.
See Example \ref{ex:discussion}, $(3)$, and \cite[Example 4.4.2]{watanabe}.


\bibliography{sample17}
\bibliographystyle{alpha}

\end{document}